\documentclass[12pt]{article}
\usepackage{epsfig,psfrag}

\usepackage{graphicx}
\usepackage{amsmath}
\usepackage{mathtools}
\usepackage{amssymb}
\usepackage{color}
\usepackage{bbm}
\usepackage{mathrsfs}
\usepackage{dsfont,accents}
\usepackage{here}
\usepackage{cite}
\usepackage{fullpage}

\DeclareMathOperator*{\rank}{rank}
\DeclareMathOperator{\eps}{\varepsilon}

\everymath{\displaystyle}

\newtheorem{theorem}{Theorem}

\newtheorem{proposition}[theorem]{Proposition}
\newtheorem{lemma}[theorem]{Lemma}
\newtheorem{define}[theorem]{Definition}
\newtheorem{example}[theorem]{Example}
\newtheorem{assumption}[theorem]{Assumption}

\newtheorem{problem}[theorem]{Problem}
\newtheorem{remark}[theorem]{Remark}

\newenvironment{proof}{{\it Proof :~}}{\hfill$\diamondsuit$\\}

\usepackage{enumitem}
\setlist[enumerate]{leftmargin=.5in}
\setlist[itemize]{leftmargin=.5in}

\title{A biology-inspired approach to the positive integral control of positive systems - the antithetic, exponential, and logistic integral controllers}

\author{Corentin Briat\thanks{NCIS, Switzerland  (corentin@briat.info, www.briat.info).}}
\date{}

\providecommand{\blue}[1]{\color{black}{#1}\color{black}\hspace{0pt}}

\begin{document}

\maketitle

\begin{abstract}
 The integral control of positive systems using nonnegative control input is an important problem arising, among others, in biochemistry, epidemiology and ecology. An immediate solution is to use an ON-OFF nonlinearity between the controller and the system. However, this solution is only available when controllers are implemented in computer systems. When this is not the case, like in biology, alternative approaches need to be explored. Based on recent research in the control of biological systems \cite{Briat:15e,Briat:16a}, we propose to develop a theory for the integral control of positive systems using nonnegative controls based on the so-called \emph{antithetic integral controller} and two \emph{positively regularized integral controllers}, the so-called \emph{exponential integral controller} and \emph{logistic integral controller}. For all these controllers, we establish several qualitative results, which we connect to standard results on integral control. We also obtain additional results which are specific to the type of controllers. For instance, we show an interesting result stipulating that if the gain of the antithetic integral controller is suitably chosen, then the local  stability of the equilibrium point of the closed-loop system does not depend on the choice for the coupling parameter, an additional parameter specific to this controller. Conversely, we also show that if the coupling parameter is suitably chosen, then the equilibrium point of the closed-loop system  is locally stable regardless the value of the gain.  For the exponential integral controller, we can show that the local stability of the equilibrium point of the closed-loop system is independent of the gain of the controller and the gain of the system. The stability only depends on the exponential rate of the controller, again a parameter that is specific to this type of controllers. Several examples are given for illustration.\\

\end{abstract}

\section{Introduction}

Integral control is undoubtedly a cornerstone of control theory as it allows for the tracking of constant reference signals while being able to reject constant disturbances, an impressive feat despite to its very simple structure. It is, therefore, not surprising to find this structure in many control algorithms for industrial processes in the so-called Proportional-Integral-Derivative (PID) controllers \cite{Astrom:95}. Interestingly, integral control is not only of engineering interest but has also been shown to arise in living organisms such as in bacteria \cite{Yi:00,Briat:15e} or mammals \cite{ElSamad:02}. Synthetic genetic regulation circuits have also recently attracted a lot of interest in the nascent field of Cybergenetics, the field pertaining on the modeling, analysis and control of biological processes using \emph{in-silico} and \emph{in vivo} controllers; see e.g. \cite{Oishi:10b,Briat:12c,Briat:13h,Chen:13,Yordanov:14,Briat:15e,DelVecchio:16,Briat:16a,DelVecchio:17,Cuba:17,Briat:17ACS}.

We propose to consider in this paper the integral control of positive systems with the constraint that the control input be nonnegative. Positive systems \cite{Farina:00} are an important class of dynamical systems whose state is confined in the nonnegative orthant. Such systems have been considered for the modeling of a wide variety of real world processes such as the modeling of populations \cite{Murray:02}, physiological systems \cite{Hovorka:04}, biochemical systems \cite{Briat:12c,Briat:13h,Briat:13i,Briat:15e,Briat:16a}, communication networks \cite{Shorten:06,Briat:13f}, etc. The control/controllability/reachability of systems with positive inputs have been well studied; see e.g. \cite{Saperstone:73,Coxson:87,Heemels:98,Leyva:14,Eden:16} and the references therein. The nonnegative proportional-integral control of linear positive systems have been studied in \cite{Briat:12c,Briat:13h,Briat:19:Opto} in the context of the control of the moments equation in biochemical reaction networks. The idea is to place between the controller and the system an ON/OFF static nonlinearity \cite{Goncalves:07} which zeroes the control input when it is negative. This makes the global stability/stabilization problem addressable in the linear system setting using, for instance, the Popov criterion \cite{Popov:61,Briat:12c,Briat:13h,Guiver:15,Briat:19:Opto} or surface Lyapunov functions \cite{Goncalves:07}.  However, the cases of nonlinear systems or stochastic systems are way more involved and there seems to be no clear way on how to address this problem in a fairly general setting.

To palliate this, we propose to use two alternative approaches. The first one is based on the use of the antithetic integral controller recently introduced in \cite{Briat:15e} in the context of the \emph{in vivo} control of single-cells and cell populations; see also \cite{Aoki:19} for an experimental validation of the approach. This controller has the benefit to implement an integral action and to naturally produce a nonnegative control input without the use of any additional static nonlinearity. The price to pay, however, is its nonlinear behavior (yet polynomial of order two) and its higher dimensionality (dimension two). However, its polynomial structure makes it interesting for then controlling polynomial/rational nonlinear systems since it makes it possible to use numerical methods such as the use of polynomial Lyapunov functions combined with optimization tools such as sum of squares programming \cite{Parrilo:00,sostools3}. This particular integral control structure works in both the deterministic and the stochastic settings. In fact, the antithetic integral controller is the simplest integral controller that is able to properly work in a stochastic setting. The performance analysis of the antithetic integral controller and its generalizations have also been theoretically studied in \cite{Qian:18,Olsman:19,Olsman:19b,Chevalier:19} in a deterministic \blue{and stochastic} settings. \blue{In those papers, the stability and robustness properties of the antithetic integral controller is studied in various parametric regimes for both the controller and system parameters, notably in the presence of degrading controller species. In particular, \cite{Chevalier:19} focuses on the extension of  those ideas to PID control.}

\blue{We show, in the linear systems case, that provided the system matrix is Hurwitz stable}, we can always find an antithetic integral controller that makes the desired equilibrium point locally exponentially stable. Compared to the standard integral controller for which only the gain $k$ needs to be tuned so that the closed-loop system is (locally) asymptotically/exponentially stable, the antithetic integral controller has an extra degree of freedom $\eta$, which we refer here to as the \emph{coupling parameter}. \blue{We show, again in the case of linear systems,} that when the gain $k$ of the integral controller lies within some interval whose bounds can be exactly characterized and accurately computed, then the closed-loop system is  (locally) asymptotically/exponentially stable regardless the choice for the coupling parameter $\eta$. \blue{This result complements existing ones \cite{Qian:18,Briat:18:Interface,Olsman:19,Olsman:19b} where the strong coupling, or strong binding, assumption (i.e. infinite coupling parameter) was used as a simplifying procedure to obtain simpler stability conditions, explicit expressions for the variance in stochastic systems or as a mean to recover disturbance rejection properties when the controller state is subject to exponential decay. However, even if in biological systems binding constants can be large, they remain finite and it is unclear whether we can extrapolate those results to the finite coupling parameter case with no proper mathematical ground. We show here that, in fact, we can extrapolate those results as long as stability is concerned since the strong binding regime is the most restrictive one in the sense that if the equilibrium system is stable when the coupling parameter is infinite then it is necessarily stable when the coupling parameter remains finite.} \blue{Still in the case of linear systems, a dual result establishing an interval of values for $\eta$ for which the closed-loop system is (locally) asymptotically/exponentially stable regardless the choice for the gain $k$. All those results naturally extend to certain classes of nonlinear systems having a unique locally exponentially stable hyperbolic equilibrium point corresponding to possible output values. Note that in this case the linearized system may not be positive which is not a problem since the Metzler structure of the system matrix is not used in any of the proof. In this regard, the obtained results therefore readily apply to the broader class of externally positive systems. Examples considering nonlinear systems coming from biochemistry and epidemiology are provided for illustration.}

The second approach is based on the use of regularizing functions, a novel flexible concept, allowing to impose certain properties to a controller such as the nonnegativity or the saturation of the control input. Regularizing functions are defined as differentiable invertible strictly increasing functions from $\mathbb{R}$ to some connected interval of $\mathbb{R}$. When the image set is a connected interval of the nonnegative orthant, the regularizing function is said to be a \emph{positively regularizing function}. When placed between the controller and the system, a positively regularizing function will make the control input nonnegative as the ON-OFF function would also do. However, the invertibility and differentiability of regularizing functions makes it possible to obtain a dynamical model for the controller and the regularizing function all together. Interestingly, when the regularizing function is chosen to be the exponential function or the logistic function, then the corresponding dynamical model is of polynomial form. This is a clear advantage over the use of an ON-OFF nonlinearity for the same reasons as for the antithetic integral control as it makes possible to consider those controllers in the nonlinear setting. \blue{It is worth mentioning that the exponential integral controller is not new and has been studied in the past, for instance, in \cite{Shoval:10,Drengstig:12,Briat:16a,Karin:16,Fejeld:17}. In particular, \cite{Shoval:10} studied its role in fold-change detection whereas \cite{Fejeld:17} focused on the controller ability to reject non-constant disturbances. In the present paper, we observe for the first time that this controller can be obtained using positively regularizing functions and we mainly focus on its stability properties and on the characterization of the different classes of systems that can be controlled using such a controller motif.} On the other hand, the logistic controller does not seem to have been studied in the past. These two controllers are expected to fail in the stochastic setting because of the presence of an absorbing state located at zero. In the deterministic case, however, the equilibrium point is shown to be structurally unstable. Analogous results to the antithetic integral controller are then obtained and illustrated through several examples. In particular, the exponential controller is shown to possess an intrinsic robustness property with respect to the gain at zero frequency of the system. The logistic integral controller does not enjoy such a property because of the presence of the saturation.\\

\noindent\textbf{Outline.} The paper is structured as follows. Positive integration and positive integral control are introduced in Section \ref{sec:PIPIC}. The control of positive linear systems using antithetic integral control is addressed in Section \ref{sec:AIC_linear} whereas the case of nonlinear positive systems is considered in Section \ref{sec:AIC_nonlinear}. The theory pertaining on the exponential and logistic integral controllers is developed in Section \ref{sec:exponential_linear} and Section \ref{sec:logistic_linear}, respectively. A concluding discussion is given in Section \ref{sec:discussion}.


\section{Positive integration}\label{sec:PIPIC}

We introduce in this section the concept of positive integration as well several positive integration structures.

\subsection{Antithetic integration}\label{sec:linpos1}

The antithetic integrator \cite{Oishi:10b,Briat:15e} is defined as
\begin{equation}\label{eq:mainsystK2}
  \begin{array}{lcl}
    \dot{z}_1(t)&=&w_1(t)-\eta z_1(t)z_2(t)\\
    \dot{z}_2(t)&=&w_2(t)-\eta z_1(t)z_2(t)\\
    z(0)&=&z_0
  \end{array}
\end{equation}
where $z:=(z_1,z_2)\in\mathbb{R}_{\ge0}^2$  is the state of the integrator, $w_1,w_2:\mathbb{R}_{\ge0}\to\mathbb{R}_{\ge0}$ are the inputs, $z_0\in\mathbb{R}_{\ge0}^2$ is the initial condition and $\eta>0$ is the coupling parameter. It is immediate to see that the state remains nonnegative whenever the initial condition and the inputs are nonnegative as well. Note also that this integrator has two states and two inputs, which is a clear difference with the standard integrator. However, it is possible to reconcile them by virtue of the following result:
\begin{proposition}
  The signal $\tilde z:=z_1-z_2$ has dynamics
  \begin{equation}
    \dot{\tilde z}(t)=\tilde w(t),\tilde z(0)=z_1(0)-z_2(0).
  \end{equation}
  and, hence, the system \eqref{eq:mainsystK2} is an implementation of the integration operator acting on the input $\tilde w:=w_1-w_2$.
\end{proposition}

This results demonstrates that even though the two inputs are nonnegative, the integrator integrates their difference, which can be negative. An interpretation of the antithetic integrator is that $z_1$ and $z_2$ contain the positive and the negative part of $\tilde z$, respectively. However, this is not purely the positive and negative parts since they are usually both nonzero.

\subsection{Regularized integration}\label{sec:linpos2}

Before defining what is regularized integration, it is necessary to define the concept of regularizing function:
\begin{define}
  We  say that $\varphi:\mathbb{R}\to\mathbb{R}$ is a regularizing function if it is
  \begin{itemize}
    \item continuously differentiable a.e., and
    \item monotonically increasing, i.e. $\varphi^\prime>0$ a.e.
  \end{itemize}
  If, moreover, it is nonnegative, then it is called a positively regularizing function.
\end{define}
The idea behind the use of regularizing functions to be combined with a standard integrator to change its properties. The definition above imposes some strong properties which are useful in the current paper and when we have applications in biology in mind. However, it is possible to drop the differentiability or the strict monotonicity assumptions so that ON-OFF nonlinearities of the form $\max\{0,\cdot\}$ or saturation functions can be considered. In the current paper, we rely on the invertibility property of regularizing functions in order to obtain a dynamical model for the combination of the regularized integration operator. This leads to the following result:
\begin{proposition}
  The dynamics of the regularized integrator $\dot{I}=w,v:=\varphi(I)$, for some regularizing function $\varphi$ and input signal $w$, is given by
  \begin{equation}
      \dot{v}(t)=\varphi^\prime(\varphi^{-1}(v(t)))w(t),v(0)=v_0
  \end{equation}
\end{proposition}
\begin{proof}
  The proof follows from the standard differentiation rules and the fact that  the inverse of $\varphi$ exists according to the definition of regularizing functions.
\end{proof}
Deriving the model for entire regularized integral controller is interesting for various reasons. First of all, it removes the regularizing function from the loop by embedding it in the model of the integrator. Another interesting fact is that this new expression for the regularized integrator emphasizes connections with existing models in biochemistry, notably in the glucose regulation system \cite{Karin:16} for a certain choice for the regularizing function. Such integrators are notably also called \emph{constrained integrators} in \cite{Xiao:18} but their correspondence with regularizing functions has never been pointed so far.

In Table \ref{tab:driving}, several correspondences between (positively) regularizing functions and the dynamical model for $v$ are provided. Interestingly, we can see that exponential-based regularizing functions lead to polynomial dynamical models for $v$, an essential property when such controllers need to be implemented in living organisms. The hyperbolic, arctangent and algebraic regularizing functions, even if not considered any further in this paper, may find applications in the control of nonlinear systems subject to input saturations. Indeed, substituting the static saturation function by a regularizing function could allow for a more systematic analysis of the closed-loop dynamics. Note also, that these regularizing functions can be made positively regularizing by simply adding a constant term to their expression.

\begin{table}[h]
\centering
\caption{Some (positively) regularizing  functions and the associated regularized integration models}\label{tab:driving}
  \begin{tabular}{|c|c|c|}
  \hline
        & $\varphi(x)$, $\alpha,\beta,\theta>0$ & $\dot{v}$ for $v=\varphi(I)$\\
        \hline
        \hline
        Exponential & $e^{\alpha x}$& $\alpha vw$\\
        Logistic & $\frac{\beta }{e^{-\alpha x}+1}$ & $\dfrac{\alpha}{\beta}v(\beta-v)w$\\
        Generalized Logistic & $\frac{\beta }{(e^{-\alpha x}+1)^\theta}$ & $\frac{\alpha\theta}{\beta^{1/\theta}}v(\beta^{1/\theta}-v)w$\\
        Hyperbolic & $\beta\tanh(\alpha x)$& $\alpha \left(1-\dfrac{v^2}{\beta^2}\right)w$\\
        Arctangent & $\beta\arctan(\alpha x)$ & $\dfrac{\alpha\beta}{1+\tan\left(v/\beta\right)^2}w$\\
        Algebraic & $\frac{\beta \alpha x}{(1+\alpha^2 x^2)^{1/2}}$& $\dfrac{\alpha}{\beta^2}(\beta^2-v^2)^{3/2}w$\\
        \hline
  \end{tabular}
\end{table}

\section{Antithetic integral control of linear internally positive systems}\label{sec:AIC_linear}

\blue{The objective of this section is to provide a thorough analysis of the antithetic integral controller and its application to the integral control of linear positive systems. First, linear positive systems are briefly recalled in Section \ref{subsec:linpos} with some of their properties. The control problem and its solution based on the use of an antithetic integral controller are defined in Section \ref{subsec:controlpb}. Local stability results are obtained in Section \ref{sec:linan1}, Section \ref{sec:linan2} and Section \ref{sec:linan3}. The perfect adaptation properties of the controller are discussed in Section \ref{subsec:PA} and the section concludes on an example inspired from biochemistry in Section \ref{subsec:ex}.}

\subsection{Linear positive systems}\label{subsec:linpos}

Let us consider the linear system
\begin{equation}\label{eq:mainsystL}
  \begin{array}{lcl}
    \dot{x}(t)&=&Ax(t)+Bu(t),\     x(0)=x_0\\
    y(t)&=&Cx(t)\\
  \end{array}
\end{equation}
where $x\in\mathbb{R}^n$ is the state of the system, $u\in\mathbb{R}$ is the control input and $y\in\mathbb{R}$ is the measured output that has to be controlled. The matrices $A\in\mathbb{R}^{n\times n}$, $B\in\mathbb{R}^{n\times 1}$ and $C\in\mathbb{R}^{1\times n}$ are the state matrix, the input matrix and the output matrix, respectively. Since we are interested in the control of linear internally positive systems, we need a way to identify this type of systems. The following result gives a complete characterization of linear internally positive systems \cite{Farina:00}:
\begin{proposition}\label{prop:posL}
The following statements are equivalent:
\begin{enumerate}[label=({\alph*})]
  \item The system \eqref{eq:mainsystL} is internally positive; i.e. for any $x_0,w(t)\ge0$ for all $t\ge0$, we have that $x(t),y(t)\ge0$ for all  $t\ge0$.
  \item The matrix $A$ is Metzler (i.e. all the off-diagonal entries are nonnegative), the matrices $B$ and $C$ are nonnegative (i.e. all the entries are nonnegative).
\end{enumerate}
\end{proposition}
A consequence of this result is that the impulse response of an internally positive system is nonnegative. However, there may exist systems which are not internally positive having a nonnegative impulse response. Those systems are called externally positive systems and only require the output to be nonnegative for any nonnegative input and zero initial conditions \cite{Farina:00}.

We also make the following assumption:
\begin{assumption}\label{hyp:1}
  The matrix $A$ is Hurwitz stable and $CA^{-1}B\ne0$.
\end{assumption}
\blue{The assumption of asymptotic stability is motivated by the fact that the integral action usually has a destabilizing effect on the dynamics of the system. It is still possible to control systems for which this assumption is not met under additional conditions on the poles and zeros of the transfer function associated with the system. This will be treated in a different article.} The second assumption simply states that the DC-gain of the system is nonzero. One implication of this assumption is that of the output controllability of \eqref{eq:mainsystL}:
\begin{proposition}\label{prop:dsdsds}
Assume that the system \eqref{eq:mainsystL} is internally positive and that Assumption \ref{hyp:1} holds. Then, the system \eqref{eq:mainsystL} is output-controllable\footnote{The system \eqref{eq:mainsystL} is output-controllable  if  $\rank\begin{bmatrix}
    CB & CAB & \ldots & CA^{n-1}B
  \end{bmatrix}=1$; see e.g. \cite{Ogata:70}.} and we have that $CA^{-1}B<0$.
\end{proposition}
\begin{proof}
For the system \eqref{eq:mainsystL} to be output controllable, it suffices that the impulse response be non-identically 0. This is equivalent to saying that its transfer function $G$ be non-identically zero as well. Since the system is internally positive and $A$ is Hurwitz stable, then we have that $\textstyle ||G||_{H_\infty}=\sup_{w\in\mathbb{R}}|G(j\omega)|=|G(0)|=|CA^{-1}B|\ne0$, by assumption. \blue{The fact that the frequency response has maximum modulus at the zero frequency comes from the fact that the system is internally positive and stable; see e.g. \cite{Briat:11h}.} Hence, the system is output-controllable. Finally, since the matrix $A$ is Metzler and Hurwitz stable then its inverse is \blue{nonpositive} \cite{Berman:94}. Hence, $CA^{-1}B<0$, since $CA^{-1}B\ne0$. The proof is completed.
\end{proof}

\subsection{The control problem and the antithetic integral controller}\label{subsec:controlpb}

We consider in this section the following control problem:
\begin{problem}\label{prob1}
  Let the reference $\mu>0$ be given and assume that the system \eqref{eq:mainsystL} is internally positive and satisfies the assumptions in Assumption \ref{hyp:1}. Find an integral controller such that
  \begin{enumerate}[label=({\alph*})]
    \item the control input $u$ is nonnegative at all times;
    \item the equilibrium point of interest of the closed-loop system consisting of the system \eqref{eq:mainsystL} and the controller is (locally) asymptotically stable;
    \item the output $y$ asymptotically tracks the reference $\mu>0$; i.e. $y(t)\to\mu$ as $t\to\infty$;
    \item the closed-loop system locally rejects constant disturbances acting on the input and on the state of the system.
    \end{enumerate}
\end{problem}
This problem and its variations have been addressed in the past using various approaches. A thoroughly considered approach relies on the use of an ON-OFF nonlinearity placed between the controller and the system; \cite{Briat:12c,Briat:13h,Guiver:15,Briat:19:Opto}. The main issue is that an ON-OFF integrator cannot be readily implemented in terms of chemical reactions and, hence, cannot be implemented inside biological organisms. Another drawback of ON-OFF nonlinearities is the difficulty of considering them for the control of nonlinear systems. On the other hand, the antithetic integral controller defined as
\begin{equation}\label{eq:mainsystK2b}
  \begin{array}{lcl}
    \dot{z}_1(t)&=&\mu-k\eta z_1(t)z_2(t)\\
    \dot{z}_2(t)&=&y(t)-k\eta z_1(t)z_2(t)\\
    u(t)&=&kz_1(t)\\
    z(0)&=&z_0\ge0,
  \end{array}
\end{equation}
where $k>0$ is the gain of the controller and $\eta>0$ is the coupling parameter, can be theoretically implemented in terms of chemical reactions. This controller was first proposed in \cite{Briat:15e} for the integral control of biological networks using \emph{in vivo} controllers. Using the implementation ideas proposed in \cite{Briat:15e}, this controller has been implemented \emph{in vivo} and convincing experimental results were obtained and reported in \cite{Aoki:19}. Even though we are not considering such implementation constraints here, the above controller has the benefit of being polynomial. In this regard, this controller may be interesting to consider when controlling polynomial or rational systems as it would allow for the use of modern optimization methods such as sum of squares programming \cite{Parrilo:00}. Note that the ON-OFF nonlinearity can be considered in the linear setting using Popov's criterion but the nonlinear setting is way more involved and there is no clear way on how to do so.  \blue{A last remark is that the model is slightly different than the one \cite{Briat:15e} where the term $\eta$ is considered instead of $k\eta$ for the polynomial term in the model. The underlying reason is purely technical and allows one to simplify the derivation of the results. The simple change of variables $\eta\leftarrow k\eta$ allows to retrieve the controller considered in \cite{Briat:15e} and to adapt the results of this paper to the controller in \cite{Briat:15e}.}

The following result shows that the closed-loop system has only one equilibrium imposed by the antithetic integral controller:
\begin{proposition}
  Let $k,\eta,\mu$ be given. Then, the equilibrium point of the closed-loop system \eqref{eq:mainsystL}-\eqref{eq:mainsystK2b} is unique and is given by
\begin{equation}\label{eq:point}
    x^*=\mu\dfrac{A^{-1}B}{CA^{-1}B},\ z_1^*=\dfrac{-\mu}{CA^{-1}Bk}\ \textnormal{and}\ z_2^*=\dfrac{-CA^{-1}B}{\eta}.
\end{equation}
\end{proposition}

The equilibrium point of the closed-loop system for some given $\mu,\eta,k>0$ is denoted by
\begin{equation}
  \mathcal{X}_{\mu,\eta,k}^*:=\left\{(x^*,z^*)\ \textnormal{defined\ in\ }\eqref{eq:point}\right\}.
\end{equation}

%

\subsection{Controller existence - a non-constructive result}\label{sec:linan1}

The following result proves the existence of an antithetic integral controller that solves the Problem \ref{prob1}:
\begin{theorem}[Local stabilization]\label{th:glabglab}
Assume that the system \eqref{eq:mainsystL} is internally positive and that it satisfies Assumption \ref{hyp:1}. Then, the following statements hold for the closed-loop system  \eqref{eq:mainsystL}-\eqref{eq:mainsystK2b}:
  \begin{enumerate}[label=({\alph*})]
    \item for any given $\mu,\eta>0$, there exists a $\bar{k}=\bar{k}(\mu,\eta)>0$ such that for any $k\in(0,\bar{k})$, the equilibrium point $\mathcal{X}_{\mu,\eta,k}^*$ is locally asymptotically stable;
    \item there exists a $\bar{k}>0$ such that, for any $k\in(0,\bar{k})$, all the equilibrium points in $\mathcal{X}_{\mu,\eta,k}^*$ for all $\eta,\mu>0$ are locally asymptotically stable.
  \end{enumerate}
\end{theorem}
\begin{proof}
\textbf{Proof of statement (a).} Let $\mu,\eta>0$ be given. The proof is based on a perturbation argument on the closed-loop system. The Jacobian linearization of the system about the equilibrium point \eqref{eq:point} is given by
  \begin{equation}\label{eq:linear}
    \begin{bmatrix}
      \dot{\tilde{x}}(t)\\
      \dot{\tilde{z}}_1(t)\\
      \dot{\tilde{z}}_2(t)
    \end{bmatrix}=\underbrace{\begin{bmatrix}
      A & Bk & 0\\
      0 & CA^{-1}Bk & \dfrac{\eta\mu}{CA^{-1}B}\\
      C & CA^{-1}Bk & \dfrac{\eta\mu}{CA^{-1}B}
    \end{bmatrix}}_{\mbox{$\tilde{A}(k)$}} \begin{bmatrix}
      \tilde{x}(t)\\
      \tilde{z}_1(t)\\
      \tilde{z}_2(t)
    \end{bmatrix}
  \end{equation}
  where $\tilde{x}=x-x^*,\tilde{z}_1=z_1-z_1^*$ and $\tilde{z}_2=z_2-z_2^*$. The above matrix $\tilde{A}(k)$ can be decomposed as $\tilde{A}_0+\tilde{A}_1k$ where
  \begin{equation}
    \tilde{A}_0=\begin{bmatrix}
      A & 0 & 0\\
      0 & 0 &\dfrac{\eta\mu}{CA^{-1}B}\\
      C & 0 & \dfrac{\eta\mu}{CA^{-1}B}
    \end{bmatrix}\quad \textnormal{and}\quad \tilde{A}_1=\begin{bmatrix}
      0 & B & 0\\
      0 & CA^{-1}B & 0\\
      0 & CA^{-1}B & 0
    \end{bmatrix}.
  \end{equation}
   \blue{The affine dependency on $k$ is a consequence of the choice for the coefficient of the second-order term in \eqref{eq:mainsystK2b}. Indeed, if $\eta$ would have been considered instead of $k\eta$, we would have obtained a dependence on $k$ and $1/k$, which would have complicated the analysis when $k$ is close to zero.}
  It is immediate to see that the spectrum of $\tilde{A}_0$, $\sigma(\tilde{A}_0)$, is simply given by
  \begin{equation}
    \sigma(\tilde{A}_0)=\sigma(A)\cup\left\{0,\dfrac{\eta\mu}{CA^{-1}B}\right\}.
  \end{equation}
  Since $CA^{-1}B<0$, then $\tilde{A}_0$ is marginally stable with a simple eigenvalue located at 0. The key idea is to show, via a perturbation argument, that the matrix $\tilde{A}_0+\eps \tilde{A}_1$ can be made Hurwitz for some sufficiently small $\eps>0$. In other words, the 0-eigenvalue of $A_0$ must shift to the open left half-plane whenever $\tilde{A}_0$ is slightly perturbed in the direction $\tilde{A}_1$. Note that the perturbation result is only an analysis tool here. Indeed, the matrix $\tilde{A}_0$ is not the matrix of the system when $k=0$ since the equilibrium point \eqref{eq:point} is not defined for $k=0$. Below, we only consider the sum $\tilde{A}_0+k\tilde{A}_1$ to be a matrix to analyze, and our main tool for doing so is perturbation theory.

  Let us then study the bifurcation of the simple eigenvalue $\lambda_0=0$ of $\tilde{A}_0$ under the perturbation $\tilde{A}_0+\eps \tilde{A}_1$. It is known that when an eigenvalue is simple, it is locally differentiable and admits the Taylor expansion \cite{Seyranian:03}
  \begin{equation}
    \lambda(\eps)=\lambda_0+\eps\lambda_1+o(\eps)
  \end{equation}
  where $\lambda_1=v_\ell \tilde{A}_1v_r$ where $v_\ell$ and $v_r$ are the left and right normalized eigenvectors (i.e. $v_\ell v_r=1$) associated with the eigenvalue $\lambda_0=0$ of the matrix $\tilde{A}_0$. They are given by
  \begin{equation}
    v_\ell=\begin{bmatrix}
      CA^{-1} & 1 & -1
    \end{bmatrix}\ \textnormal{and}\ v_r=\begin{bmatrix}
      0&   1&     0
    \end{bmatrix}^T.
  \end{equation}
  Using the above vectors, we get that $\lambda_1=CA^{-1}B$, which is negative by Proposition \ref{prop:dsdsds}. Therefore, a small positive $\eps$ will shift the eigenvalue $\lambda_0=0$ to the open left-half plane. As a result, using the continuity property of eigenvalues, there will exist a (sufficiently small) $\bar{k}=\bar{k}(\mu,\eta)$ such that for all $k\in(0,\bar{k})$, the equilibrium point $\mathcal{X}_{\mu,\eta,k}^*$ is locally exponentially stable. This proves statement (a).\\

  \noindent \textbf{Proof of statement (b).} Using the fact that the perturbation direction $\lambda_1=CA^{-1}B<0$ is independent of $\eta$ and $\mu$, this means that it is possible to find a $\bar{k}$ independent of $\eta$ and $\mu$ such that for all $k\in(0,\bar{k})$, all the equilibrium points $\mathcal{X}^*_{\mu,\eta,k}$ for all $\eta,\mu>0$ are locally exponentially stable. This proves statement (b).
\end{proof}

\begin{example}\label{ex:example1}
  Let us illustrate this result through a simple example. We consider here the system \eqref{eq:mainsystL} with the matrices
  \begin{equation}
    A=\begin{bmatrix}
      -1 & \blue{0}\\
      \blue{1} & 1
    \end{bmatrix},B=\begin{bmatrix}
      1\\0
    \end{bmatrix},C=\begin{bmatrix}
      0\\ 1
    \end{bmatrix}^T.
  \end{equation}
  The matrix $\tilde{A}(k)$ of the linearized system around the equilibrium point $(\mu,\mu,-\mu/k,1/\eta)$ is given by
  \begin{equation}
    \tilde{A}(k)=\begin{bmatrix}
      -1 & \blue{0} & k & 0\\
      1 & -1 & 0 & 0\\
      0 & 0 & -k & -\eta\mu\\
      0 & 1 & -k & -\eta\mu
    \end{bmatrix}.
  \end{equation}
  The characteristic polynomial $\chi(\lambda)$ associated with this matrix is given by
  \begin{equation}
    \chi(\lambda) := \lambda^4+(2+k+\eta \mu)\lambda^3+(2k+1+2\eta\mu)\lambda^2+(k+\eta\mu)\lambda+k\eta\mu.
  \end{equation}
  The Routh-Hurwitz stability criterion gives the following conditions for the stability of the matrix $\tilde{A}(k)$
  \begin{equation}
    2+k+\eta\mu>0, k\eta\mu>0, \dfrac{2(k+\eta\mu+1)^2}{2+k+\eta\mu}>0\ \textnormal{and } k\eta\mu(2+k+\eta\mu)-\dfrac{2(k+\eta\mu)(k+\eta\mu+1)^2}{2+k+\eta\mu}.
  \end{equation}
  The first three conditions are trivially satisfied. Expanding the last one gives the following condition
  \begin{equation}
    \mathcal{C}(k,\eta\mu):=\eta^3\mu^3(k-2)+2\eta^2\mu^2(k^2-k-2)+\eta\mu(k^3-2k^2-6k-2)-2k(1+k)^2<0.
  \end{equation}
  Noting that $\mathcal{C}(0,\eta\mu)=-2\eta^3\mu^3-\eta^2\mu^2-2\eta\mu<0$ (but the matrix is not Hurwitz stable because $k\eta\mu=0$) and using the fact that $\mathcal{C}(k,\eta\mu)$ is continuous in $k$, we immediately get the existence of a $\bar{k}(\mu,\eta)>0$, such that $\mathcal{C}(k,\eta\mu)$ and all the other conditions are negative for all $k\in(0,\bar{k}(\mu,\eta))$. For instance, if we pick $\eta=1$ and $\mu=10$, we get that
  \begin{equation}
    \mathcal{C}(k,10)=8k^3+176k^2+738k-2420<0.
  \end{equation}
  Interpreting now the above expression as a polynomial in $k$, we can see that there is one sign change in the coefficients. By virtue of the Descartes' rule of signs, one can conclude that there is one and only one positive root to this polynomial, which we denote by $\bar{k}(10,1)$. Since, the left-hand side of the condition is negative for $k=0$, then we get that the inequality holds for all $k\in(0,\bar{k}(10,1))$. Numerical calculations show that $\bar{k}(10,1)$ is approximately 2.11. This illustrates the first statement of Theorem \ref{th:glabglab}.

  We now illustrate the second statement of Theorem \ref{th:glabglab}. We then view the condition  $\mathcal{C}(k,\eta\mu)<0$ as a polynomial in $\eta\mu$. Obviously, if $k>0$ is chosen such that all the coefficients of the polynomial are negative then we have that $\mathcal{C}(k,\eta\mu)<0$ for all $\eta\mu>0$. In fact, by virtue of Descartes' rule of signs, this condition is a necessary and sufficient condition for the negativity of the polynomial for all $\eta\mu>0$. We obtain the following conditions on $k$:
  \begin{equation}
    k<2, k^2-k-2<0\textnormal{ and }k^3-2k^2-6k-2<0.
  \end{equation}
  For $k=0$, the above conditions are satisfied. From Descartes' rule of sign, we also know that each polynomial has one and only one positive solution. Interestingly, $k=2$ is the positive root of the second polynomial whereas the third polynomial evaluated at $k=2$ gives -16. Hence, its positive root is larger than 2. As a conclusion, if we pick $k\in(0,2)$, then $\mathcal{C}(k,\eta\mu)<0$ holds for all $\eta\mu>0$. This illustrates the second statement.
\end{example}

\subsection{A constructive result}\label{sec:linan2}

Theorem \ref{th:glabglab} is non-constructive in nature. The calculations in the example show that one can compute important parameters for the closed-loop system. The downside is that the calculations will not scale very well as the dimension of the system increases. In this regard, it would be interesting to derive constructive results. The first step towards this goal is to show that for any $\mu,k>0$, we can find a small enough $\eta>0$ such that the equilibrium point of the closed-loop system is locally exponentially stable. This is formally stated below:
\begin{lemma}\label{lem:etasmall}
  Assume that the system \eqref{eq:mainsystL} is internally positive and that it satisfies Assumption \ref{hyp:1}. Assume further that $k,\mu>0$ are given. Then, there exists $\bar{\eta}=\bar{\eta}(\mu)>0$ such that for any $\eta\in(0,\bar{\eta})$, the equilibrium point $\mathcal{X}^*_{\mu,\eta,k}$ is locally asymptotically stable.
\end{lemma}
\begin{proof}
Similarly to as in the proof of Theorem \ref{th:glabglab}, we first rewrite the linear system matrix as
  \begin{equation}
    \begin{bmatrix}
      A & Bk & 0\\
      0 & CA^{-1}Bk & 0\\
      C & CA^{-1}Bk & 0
    \end{bmatrix}+\eta\begin{bmatrix}
      0 & 0 & 0\\
      0 & 0 & \dfrac{\mu}{CA^{-1}B}\\
      0 & 0 & \dfrac{\mu}{CA^{-1}B}
    \end{bmatrix}.
  \end{equation}
  The spectrum of the matrix to the left is given by $\lambda(A)\cup\{0,CA^{-1}Bk\}$, which corresponds to the spectrum of a marginally stable matrix. We use, once again, a perturbation argument on the 0-eigenvalue. The corresponding normalized left- and right-eigenvectors are given by
  \begin{equation}
    u_\ell=\begin{bmatrix}
      -CA^{-1} & 0 & 1
    \end{bmatrix}\ \textnormal{and}\  u_r=\begin{bmatrix}
      0 & 0 & 1
    \end{bmatrix}^T.
  \end{equation}
  The eigenvalue bifurcates according to $\lambda(\eta)=\lambda_0+\eta\lambda_1+o(\eta)$ where $\lambda_1=\frac{\mu}{CA^{-1}B}$. Therefore, we can conclude on the existence of a $\bar{\eta}$ such that for any $\eta\in(0,\bar{\eta})$, the equilibrium points $\mathcal{X}^*_{\mu,\eta,k}$ are locally asymptotically stable. Note, moreover, that this result does not depend on the choice for $k$.
\end{proof}

Whereas the above result states that, for any $k,\mu>0$, we can find a small enough $\eta>0$, that makes the unique equilibrium point of the closed-loop system locally exponentially stable, the next one provides a condition on the gain $k$ under which the unique equilibrium point of the closed-loop system locally exponentially stable provided that $\eta>0$ is large enough.
\begin{lemma}\label{lem:etalarge}
Assume that the system \eqref{eq:mainsystL} is internally positive and that it satisfies Assumption \ref{hyp:1}. Assume further that $k$ is chosen such that the matrix
 \begin{equation}\label{eq:dsqpodsjdjsd}
  \mathcal{M}(k):= \begin{bmatrix}
     A & Bk\\
     -C & 0
   \end{bmatrix}
 \end{equation}
 is Hurwitz stable. Then, there exists $\bar{\eta}=\bar{\eta}(k)>0$ such that for any $\eta>\bar{\eta}$ the equilibrium point $\mathcal{X}_{\mu,\eta,k}^*$ is locally exponentially stable.
\end{lemma}
\begin{proof}
  Let $\eps=1/\eta$, and rewrite the matrix of the linearized system as
  \begin{equation}\label{eq:dksqldso}
\begin{bmatrix}
      0 & 0 & 0\\
      0 & 0 & \dfrac{\mu}{CA^{-1}B}\\
      0 & 0 & \dfrac{\mu}{CA^{-1}B}
    \end{bmatrix}+\begin{bmatrix}
      A & Bk & 0\\
      0 & CA^{-1}Bk & 0\\
      C & CA^{-1}Bk & 0
    \end{bmatrix}\eps.
  \end{equation}
  The case when $\eps=0$ corresponds to the case $\eta=\infty$. So, in this case, we actually perturb from infinity. The matrix to the left has spectrum $\left\{\dfrac{\mu}{CA^{-1}B},0\right\}$ where the 0-eigenvalue has multiplicity $n+1$. The left- and right-eigenvectors are given by
  \begin{equation}
    v_\ell=\begin{bmatrix}
      I_n & 0 & 0\\
      0 & 1 & -1
    \end{bmatrix}\ \textnormal{and}\ v_r=\begin{bmatrix}
      I_n & 0 & 0\\
      0 & 1 & 0
    \end{bmatrix}^T.
  \end{equation}
  Since the rank of the eigenvectors is $n+1$, hence the 0-eigenvalue is semisimple\footnote{An eigenvalue is semisimple when the algebraic multiplicity equals the geometric multiplicity.}. Therefore it smoothly bifurcates under the action of the perturbation parameter $\eps>0$ according to the expression
  \begin{equation}
    \lambda^i(\eps)=0+\eps\lambda_1^i+o(\eps)
  \end{equation}
  where the $\lambda_1^i$'s are the eigenvalues of the matrix
  \begin{equation}
    v_\ell\begin{bmatrix}
      A & Bk & 0\\
      0 & CA^{-1}Bk & 0\\
      C & CA^{-1}Bk & 0
    \end{bmatrix}v_r= \begin{bmatrix}
     A & Bk\\
     -C & 0
   \end{bmatrix}=\mathcal{M}(k).
    \end{equation}
  Therefore, if $k>0$ is chosen such that the above matrix is Hurwitz stable, then the 0-eigenvalues of the left-matrix of \eqref{eq:dksqldso} bifurcate in $n+1$ (distinct or not) eigenvalues that are located in the open left-half plane. The proof is complete.
\end{proof}

The above result shows that if $k$ is chosen such that the matrix $\mathcal{M}(k)$ is Hurwitz stable, then the equilibrium point is locally asymptotically stable for any sufficiently large $\eta$'s. However, this matrix is nothing else but the closed-loop system matrix of the system \eqref{eq:mainsystL} controlled with a standard integral controller. This demonstrates a natural connection between antithetic integral control and standard integral control. \blue{The limiting case $\eta\to\infty$, called "strong binding regime" in \cite{Olsman:19,Olsman:19b}, was used there as a simplifying assumption in order to get simpler stability condition. This assumption was supported by the fact that strong binding affinity occurs naturally in living organisms.  The differences between those results and the ones proposed in this paper lies at the level of generality, the considered mathematical tools, and, as we shall see below, much stronger results on the robustness of the stability of the equilibrium point with respect to $\eta$ or $k$ can be obtained. Other similar results have been obtained in \cite{Qian:18,Qian:19} using singular perturbation theory to account for a very large coupling parameter value. The approach is quite general and addresses different problems such as the approximation of a system into a reduced-order one and bounding trajectories between the two systems in order to conclude on perfect adaptation property for leaky antithetic integral controllers. In this paper, stability of the system is simply assumed.}

We are now in position to state the main result that complements those in \cite{Qian:18,Qian:19,Olsman:19,Olsman:19b}
\begin{theorem}\label{eq:mainconstructive}
Assume that the system \eqref{eq:mainsystL} is internally positive and that it satisfies Assumption \ref{hyp:1}.  Define $\bar{k}_\infty$ as
  \begin{equation}
    \bar{k}_\infty:=\sup\left\{\kappa\ge0\ \textnormal{s.t.}\ \mathcal{M}(\kappa):=\begin{bmatrix}
      A & B\kappa\\
      -C & 0
    \end{bmatrix}\ \textnormal{Hurwitz}\right\}.
  \end{equation}
  When the matrix is Hurwitz for all $\kappa>0$, then we set $\bar k_\infty=\infty$.

  Then, for all $k\in(0,\bar{k}_\infty)$, the equilibrium point $\mathcal{X}_{\mu,\eta,k}^*$ is locally asymptotically stable for all $\eta,\mu>0$.
\end{theorem}
\begin{proof}
We have also proved in Lemma \ref{lem:etasmall}, that there exist small enough values for $\eta$ for which the equilibrium points are locally asymptotically stable. Similarly, we have proved in Lemma \ref{lem:etalarge} that there exist large enough values for $\eta$ for which the equilibrium points are locally asymptotically stable. The idea is to prove that if $k$ is chosen accordingly, the equilibrium points are locally asymptotically stable for all $\eta\in(0,\infty)$.

We have the following necessary condition that $k$ must be chosen such that $\mathcal{M}(k)$ be Hurwitz stable. We prove that here it that is also sufficient. To this aim, let $\theta=-\dfrac{\mu\eta}{CA^{-1}B}>0$ and, in this case, the Jacobian matrix rewrites
 \begin{equation}
   \Psi(\theta):=\begin{bmatrix}
      A & Bk & 0\\
      0 & CA^{-1}Bk & -\theta\\
      C & CA^{-1}Bk & -\theta
    \end{bmatrix}.
 \end{equation}
 The characteristic polynomial of the above matrix can be shown to be equal to
 \begin{equation}
   \det(sI-\Psi(\theta))=\det(sI-A)\left[s(s-CA^{-1}Bk)+\theta(H(s)k+s)\right]
 \end{equation}
 where $H(s):=C(sI-A)^{-1}B$ and where we have used the Schur determinant formula. Since $A$ is Hurwitz, we just have to study the distribution of the zeros of the second factor. Since the poles and the zeros are located at 0 and the open left half-plane, then their sum
 \blue{\begin{equation}
   F(s,\theta):=s(s-CA^{-1}Bk)+\theta(H(s)k+s)
 \end{equation}
 and decompose $H(s)=N(s)/D(s)$ where $N(s),D(s)$ are polynomials. Therefore the zeros of $F(s,\theta)$ are also the zeros of
 \begin{equation}
    \tilde F(s,\theta):=s(s-CA^{-1}Bk)D(s)+\theta(N(s)k+sD(s)).
 \end{equation}
  From the above expression, if we view the roots of the $ \tilde F(s,\theta)$ as a function of $\theta$, those roots start at those of $s(s-CA^{-1}Bk)D(s)$ when $\theta=0$ and end at those of $N(s)k+sD(s)$ when $\theta=\infty$. Note that $N(s)k+sD(s)$ coincides with the denominator of the transfer function describing the closed-loop system consisting of the system $H(s)$ controlled by the integral controller $k/s$. We know that $N(s)k+sD(s)$ is stable polynomial since $\mathcal{M}(k)$ is Hurwitz stable. We also know that for a small enough $\theta$, $ \tilde F(s,\theta)$ is also a stable polynomial. It remains to characterize the behavior of those roots as $\theta$ increases from 0 to infinity. The difficulty here lies in the fact that some of the roots will escape to infinity following some asymptotes due to the fact that the polynomials $s(s-CA^{-1}Bk)D(s)$ and  $N(s)k+sD(s)$ have different degrees. To do so, we will use a root locus argument \cite{Haidekker:13} and, to this aim, we define the transfer function
 \begin{equation}
 \begin{array}{lcl}
      G(s)&:=&\dfrac{H(s)k+s}{s(s-CA^{-1}Bk)}\\
      &=&\dfrac{N(s)k+D(s)s}{s(s-CA^{-1}Bk)D(s)}.
 \end{array}
 \end{equation}
  The root locus analysis states that if $n_z$ and $n_p$ are the number of zeros and poles of the above transfer function, then one has $n_p-n_z$ asymptotes that intersect the real axis at the point $$\chi:=\dfrac{\sum_{i=1}^{n_p}p_i-\sum_{i=1}^{n_z}z_i}{n_p-n_z}$$ where $p_i$ and $z_i$ are the poles and zeros, counting multiplicity and leaves this point with angle
 \begin{equation}
   \varphi_i=\dfrac{\pi+2(i-1)\pi}{n_p-n_z},i=1,\ldots,n_p-n_z.
 \end{equation}
  The polynomial $s(s-CA^{-1}Bk)D(s)$ has $n+2$ roots where $n$ denotes the order of the polynomial $D(s)$, so we have that $n_p=n+2$. Similarly, $N(s)k+sD(s)$ has $n+1$ roots since the degree of $N(s)$ is at most that of $D(s)$ and, hence, $n_z=n+1$. Therefore, there is only one asymptote leaving the point $\chi$ with angle $\varphi_1=\pi$. Since there is no direct-feedthrough in the system (the input does not directly influence the output) then we have that the order of the polynomial $N(s)$ is strictly less than $n$. From Vieta's formulas, we have that the sum of the roots of $kN(s)+sD(s)=0$ is equal to $-d_{n-1}$ where $d_{n-1}$ is the coefficient of the polynomial $D(s)$ associated with the $n-1$-th power. Similarly, the sum of the zeros of $s(s-CA^{-1}Bk)D(s)$ is equal to $CA^{-1}Bk-d_{n-1}$. Therefore, $$\chi=CA^{-1}Bk-d_{n-1}+d_{n-1}=CA^{-1}Bk<0.$$ As a result, the only escaping root to infinity escapes to $-\infty$ along the horizontal axis and, therefore, the roots of the polynomial $ \tilde F(s,\theta)$ are located in the open left half-plane for all $\theta\in(0,\infty)$. This proves the result.}
\end{proof}

\blue{This result is important for multiple reasons. First of all, it is a result that shows that $\eta$ can be freely chosen as long as $k\in(0,\bar k_\infty)$. This parameter can be used to ensure additional properties for the closed-loop system such as the settling time. However, what this result states is that the strong-binding regime (i.e. $\eta=\infty$) is the worst case regime in terms of stability. Ensuring stability in this regime is sufficient to ensure stability for any other binding regime. This is particulary important since it means that the results obtained in \cite{Olsman:19,Olsman:19b} remain valid even when the coupling parameter is finite, which is likely to be the case in practice.} \blue{Perhaps surprisingly, a dual result can be found in which the gain $k$ can be made free by suitably choosing $\eta$. This result is stated below:}
\blue{\begin{theorem}\label{eq:mainconstructive2}
Assume that the system \eqref{eq:mainsystL} is internally positive and that it satisfies Assumption \ref{hyp:1}. Define $\bar{\eta}_\infty$ as
  \begin{equation}
    \bar{\eta}_\infty:=\dfrac{g^2}{\bar{\mu}}\sup\left\{\kappa\ge0\ \textnormal{s.t.}\ \mathcal{M}(\kappa):=\begin{bmatrix}
      A & B\kappa\\
      -C & 0
    \end{bmatrix}\ \textnormal{Hurwitz}\right\}
  \end{equation}
  where $0<\mu\le\bar{\mu}$. When the matrix is Hurwitz for all $\kappa>0$, then we set $\bar{\eta}_\infty=\infty$.

  Then, for all $\eta\in(0,\bar{\eta}_\infty)$, the equilibrium point $\mathcal{X}_{\mu,\eta,k}^*$ is locally asymptotically stable for all $k>0$ and all $\mu\in(0,\bar{\mu})$.
\end{theorem}
\begin{proof}
  The proof follows from the same lines as the proof of Theorem \ref{eq:mainconstructive} and is therefore only sketched. The starting point is again the characteristic polynomial $ \tilde F(s,\theta)$ which we rewrite as
  $$\tilde F(s,\theta)= (s+\theta)sD(s)+k(gsD(s)+\theta N(s))$$ where $g=-CA^{-1}B$ is the gain of the system. Interestingly, we can see that when $k=0$, the roots of that polynomial are those of $(s+\theta)sD(s)$ whereas when $k\to\infty$ the roots tend to those of $gsD(s)+\theta N(s)$. We know that for small enough $k$'s, the polynomial is stable regardless the values for $\mu,\eta>0$. The roots of  $gsD(s)+\theta N(s)$ are in the open left half-plane for all  $\mu\in(0,\bar{\mu})$ if and only if $\eta\in(0,\bar{\eta}_\infty)$. To show that the system remains stable for all values for $\eta$, we again rely on a root locus argument and note that the degrees of the numerator and the denominator are again equal to $n+1$ and $n+2$, so the only asymptote coincides with the real axis and point towards $-\infty$. The sums of zeros and poles are given by $-gd_{n-1}$ and $-\mu\eta/g-d_{n-1}$, respectively. In this regard, the asymptote starts at $\chi=-\mu\eta/g-d_{n-1}+gd_{n-1}$ and points towards $-\infty$ along the real axis. Therefore, the roots of the polynomial $\tilde F(s,\theta)$ lie in the open left half-plane for all $k\in(0,\infty)$ provided that $\eta\in(0,\bar{\eta}_\infty)$.
\end{proof}}

\begin{example}
  Let us consider again the system of Example \ref{ex:example1}. After tedious calculations, it was shown that if $k<\bar{k}_\infty=2$, then the closed-loop system is locally asymptotically stable for all $\eta,\mu>0$. We now prove the same result using Theorem \ref{eq:mainconstructive}. We first form the matrix
  \begin{equation}
    \begin{bmatrix}
      A & Bk\\-C & 0
    \end{bmatrix}=\begin{bmatrix}
      -1 & 0 &k\\
      1 & -1 &0\\
      0 & -1 & 0
    \end{bmatrix}.
  \end{equation}
  The corresponding characteristic polynomial is given by
  \begin{equation}
    \lambda^3+2\lambda^2+\lambda+k.
  \end{equation}
  The Routh-Hurwitz stability criterion yields the conditions $k>0$ and $k-2<0$. Hence, $\bar{k}_\infty=2$. Similarly, we have that $\bar{\eta}_\infty=2g^2/\mu$.
\end{example}

\subsection{Computing $\overline{k}_\infty$ and $\overline{\eta}_\infty$}\label{sec:linan3}

We propose in this section some ways to establish the value for $\bar k$. The first approach allows one to determine whether this value is finite using the concept of strictly positive real transfer functions and strictly passive systems. The second approach is based on stability crossing where we study the existence of purely imaginary eigenvalues which essentially reduces to the analysis of some polynomials.\\

\noindent\textbf{The case $\boldsymbol{\bar{k}_\infty=\infty}$ and $\boldsymbol{\bar{\eta}_\infty=\infty}$.} Interestingly, there are cases where the equilibrium point of the closed-loop system  \eqref{eq:mainsystL}-\eqref{eq:mainsystK2b} is locally exponentially stable for any $\mu,\eta,k>0$. This is formalized in the result below:

\begin{theorem}\label{th:passive}
Assume that the system \eqref{eq:mainsystL} is internally positive and that it satisfies Assumption \ref{hyp:1}. Define further its transfer function as $G(s):=C(sI-A)^{-1}B$. Then, the following statements are equivalent:
\begin{enumerate}[label=({\alph*})]
   \item The system \eqref{eq:mainsystL} is strictly passive.
   \item The transfer function $G$ is strictly positive real, that is, $\Re[G(j\omega)]>0$ for all $\omega\in\mathbb{R}_{\ge0}$ and $$\lim_{\omega\to\infty}\omega^2\Re[G(j\omega)]>0.$$
  \item There exist symmetric positive definite matrices $P,Q$ such that $A^TP+PA=-Q$ and $PB=C^T$.
  \item The equilibrium point of the closed-loop system  \eqref{eq:mainsystL}-\eqref{eq:mainsystK2b} is locally exponentially stable for any $\mu,\eta,k>0$.
\end{enumerate}
\end{theorem}
\begin{proof}
\blue{The proof of the equivalence between two first statements can be found in \cite{Kottenstette:14}. A proof for the statement (b) can be also found in \cite{Tao:90}. The equivalence with the third statement comes from the  Kalman-Yakubovich-Popov Lemma; see e.g \cite{Khalil:02}. In this regard, we simply need to prove the equivalence with the last statement. The proof  follows from an application of the Nyquist stability criterion. First note that, since $-CA^{-1}B>0$, then $\Re[G(j\omega)]>0$ for all $\omega\in\mathbb{R}_{\ge0}$ is equivalent to saying that $\arg[G(j\omega)]\in(-\pi/2,\pi/2)$ for all $\omega\in\mathbb{R}_{\ge0}$. The loop-transfer associated with the system described by $\mathcal{M}(k)$ is given by $L(s)=kG(s)/s$ and we have that}
\begin{equation}
\begin{array}{rcl}
    \arg(L(j\omega))  &=&    \arg(kG(j\omega))-\arg(j\omega)\\
                                    &=&     \arg(G(j\omega))-\pi/2.
\end{array}
\end{equation}
and
\begin{equation}
|L(j\omega)|=\dfrac{k|G(j\omega)|}{\omega}.
\end{equation}
Since the system $G(s)$ is stable then, from the Nyquist stability criterion, the closed-loop system is unstable if and only if there is an $\omega_c>0$ such that $\arg(L(j\omega_c))=-\pi$ and $k\ge k_c$ where $k_c=\omega_c/|G(j\omega_c)|$; i.e. the Nyquist plot encircles at least once the critical point $-1$. For such an $\omega_c$ to exist, we need that  $\arg(G(j\omega_c))=-\pi/2$. Therefore, a necessary and sufficient condition for this  $\omega_c$ to not exist is that $\Re[G(j\omega)]>0$ for all $\omega\in\mathbb{R}_{\ge0}$. This concludes the proof.
\end{proof}

\blue{We give below an example of a controlled reaction network satisfying such a condition.
\begin{example}
  Let us consider a reaction network represented by the following linear system which is inspired from an example in \cite{Olsman:19}
  \begin{equation}
    \begin{array}{rcl}
      \dot{x}(t)&=&\begin{bmatrix}
        -\gamma & k_1\\
        k_2 & -\gamma
      \end{bmatrix}x(t)+\begin{bmatrix}
        0\\1
      \end{bmatrix}\\
      y(t)&=&\begin{bmatrix}
        0 & 1
      \end{bmatrix}
    \end{array}
  \end{equation}
  where $\gamma^2-k_1k_2>0$. This system is internally positive with a Hurwitz stable system matrix and we have that $-CA^{-1}B=\gamma/(\gamma^2-k_1k_2)$. The associated transfer function is given by
  \begin{equation}
    H(s)=\dfrac{s+\gamma}{s^2+2\gamma s+\gamma^2-k_1k_2}.
  \end{equation}
  Clearly the relative degree is equal to one and we have that
  \begin{equation}
    \Re[H(j\omega)]=\dfrac{\gamma\omega^2+\gamma^2-k_1k_2}{(-\omega^2+\gamma^2-k_1k_2)^2+4\gamma^2\omega^2}.
  \end{equation}
  Since $\gamma^2-k_1k_2>0$, the numerator is always positive and we have that
  $$\lim_{\omega\to\infty}\omega^2\Re[G(j\omega)]=\gamma>0.$$
  Hence, the transfer function is strictly positive real and the equilibrium point of the closed-loop system will be locally exponentially stable for all positive controller parameters $\mu,k$ and $\theta$.

  Alternatively, we can check the condition of statement (c) in Theorem \ref{th:passive}. Let us define
  \begin{equation}
    P=\begin{bmatrix}
      p_1 & p_2\\
      p_2 & p_3
    \end{bmatrix}.
  \end{equation}
  Then, the condition that $PB-C^T=0$ yields $p_2=0$ and $p_3=1$. Moreover, we have that
  \begin{equation}
    A^TP+PA=\begin{bmatrix}
      -2p_1\gamma & p_1k_1+k_2\\
      p_1k_1+k_2 & -2\gamma
    \end{bmatrix}.
  \end{equation}
  This matrix is negative definite if and only if the trace is negative and the determinant is positive. As the trace is negative for all $p_1>0$, we need to find a suitable value for $p_1>0$ such that the determinant is positive. This yields the condition
  \begin{equation}
    k_1^2p_1^2+2p_1(k_1k_2-2\gamma^2)+k_2^2<0.
  \end{equation}
  This condition is minimum for $p_1=\dfrac{2\gamma^2-k_1k_2}{k_1^2}$ which yields
  \begin{equation}
       A^TP+PA=\dfrac{2\gamma}{k_1^2}\begin{bmatrix}
      -(2\gamma^2-k_1k_2)& k_1\gamma\\
      k_1\gamma & -k_1^2.
    \end{bmatrix}
  \end{equation}
  The determinant of the matrix without the factor is then equal to $k_1^2(\gamma^2-k_1k_2)$ and is positive. This proves that the system is strictly passive.
\end{example}}

\noindent\textbf{The case $\boldsymbol{\bar{k}_\infty<\infty}$ and $\boldsymbol{\bar{\eta}_\infty<\infty}$ .} When the test previously presented fails, then $\bar k$ is necessarily finite and its value can be computed using a stability crossing test, a popular method in the time-delay systems community \cite{Niculescu:01,GuKC:03}. This is stated in the result below:
\blue{\begin{proposition}\label{prop:k_bar}
Assume that the system \eqref{eq:mainsystL} is internally positive and that it satisfies Assumption \ref{hyp:1}. Let $P(s,k)$ be the characteristic polynomial of the matrix $\mathcal{M}(k)$ in \eqref{eq:dsqpodsjdjsd} and define further the real polynomials $P^0_R(\omega),P^1_R(\omega),P^0_I(\omega)$ and $P^1_I(\omega)$  as
\begin{equation}
  P^0_R(\omega)+kP^1_R(\omega):=\Re[P(k,j\omega)] \textnormal{ and } P^0_I(\omega)+kP^1_I(\omega):=\Im[P(k,j\omega)].
\end{equation}
If the set
\begin{equation}\label{eq:polycrtic}
\Omega:=\left\{\omega>0:P_R^1(\omega)P^0_I(\omega)-P_R^0(\omega)P_I^1(\omega)=0\right\}
\end{equation}
is nonempty, then the value for $\bar{k}_\infty$ and $\bar\eta_\infty$ are given by
  \begin{equation}\label{eq:bar k}
  \bar{k}_\infty=\inf_{\bar\omega\in\Omega}\left\{\begin{array}{lcl}
    -\dfrac{P_R^0(\bar\omega)}{P_R^1(\bar\omega)}&\textnormal{if}&P_R^1(\bar\omega)\ne0\\
    -\dfrac{P_I^0(\bar\omega)}{P_I^1(\bar\omega)}&\textnormal{if}&P_I^1(\bar\omega)\ne0.
  \end{array}\right.
\end{equation}
  \begin{equation}\label{eq:bar eta}
  \bar{\eta}_\infty=\dfrac{g^2}{\mu}\inf_{\bar\omega\in\Omega}\left\{\begin{array}{lcl}
    -\dfrac{P_R^0(\bar\omega)}{P_R^1(\bar\omega)}&\textnormal{if}&P_R^1(\bar\omega)\ne0\\
    -\dfrac{P_I^0(\bar\omega)}{P_I^1(\bar\omega)}&\textnormal{if}&P_I^1(\bar\omega)\ne0.
  \end{array}\right.
\end{equation}
where $g=-CA^{-1}B$. When the set $\Omega$ is empty, then $\bar{k}_\infty=\infty$ and $\bar{\eta}_\infty=\infty$.
\end{proposition}}
\begin{proof}
\blue{The matrix $\mathcal{M}(k)$ is Hurwitz stable if and only if the roots of its characteristic polynomial $P(s,k)$ are all located in the open left-half plane. We then look for pairs $(\bar \omega,\bar k)\in\mathbb{R}^2_{>0}$ such that $P(j\bar \omega,\bar k )=0$. By doing so, we look for critical values of $\bar k $ for which we necessarily have a pair of eigenvalues on the imaginary axis. This expression can be rewritten as
  \begin{equation}
    \left[P^0_R(\bar \omega)+\bar kP^1_R(\bar \omega)\right]+j\left[P^0_I(\bar \omega)+\bar kP^1_I(\bar \omega)\right]=0
  \end{equation}
  where  $P^0_R(\omega)+kP^1_R(\omega):=\Re[P(j\omega,k)]$, $P^0_I(\omega)+kP^1_I(\omega):=\Im[P(j\omega,k)]$. If such a pair $(\bar k,\bar \omega)$ exists then we have that $P^0_R(\bar \omega)+\bar kP^1_R(\bar \omega)=P^0_I(\bar \omega)+\bar kP^1_I(\bar \omega)=0$. This can be rewritten as
  \begin{equation}
    \begin{bmatrix}
      P^0_R(\bar \omega) & P^1_R(\bar \omega)\\
      P^0_I(\bar \omega) & P^1_I(\bar \omega)
    \end{bmatrix} \begin{bmatrix}
        1\\
        \bar{k}
    \end{bmatrix}=0.
  \end{equation}
  This is equivalent to say that the vector lies in the kernel of the matrix and a necessary and sufficient condition for that is that the matrix be singular, or, equivalently, that its determinant be equal to zero. This leads to the condition \eqref{eq:polycrtic}. We can then solve for $\bar\omega$ and either use $P^0_R(\bar \omega)+\bar kP^1_R(\bar \omega)=0$ or $P^0_I(\bar \omega)+\bar kP^1_I(\bar \omega)=0$ to find $\bar k$. Since, we may have multiple solutions for the crossing frequencies, we need to choose the smallest $\bar k$ to ensure the stability. The result follows.}
\end{proof}

\subsection{Disturbance rejection/Perfect adaptation}\label{subsec:PA}

Let us analyze the disturbance rejection properties of the antithetic integral controller. \blue{It is expected that this controller rejects constant disturbances on the control input and on the states of the system by virtue of the internal model principle which stipulates that ``\textit{any good regulator must create a model of the dynamic structure of the environment in the closed-loop system}'' \cite{Bengtson:77}. This follows from the fact that the integrator models constant disturbances. Disturbance rejection of the antithetic motif was notably addressed in \cite{Briat:15e} in the stochastic setting with respect to constant disturbance rejection. A more general analysis in the deterministic setting is described in \cite{Olsman:19} using the sensitivity function. We show this using a different approach. To this aim, let us consider the following disturbed system}
\begin{equation}\label{eq:mainsystLdist}
  \begin{array}{lcl}
    \dot{x}(t)&=&Ax(t)+Bu(t)+Ed,x(0)=x_0\\
    y(t)&=&Cx(t)\\
  \end{array}
\end{equation}
where the disturbance vector $d\in\mathbb{R}_{\ge0}$ has been added and where $E\in\mathbb{R}^n_{\ge0}$.

Assuming that the system is internally positive and that it satisfies the conditions of Assumption \ref{hyp:1}, then for any constant $u\ge0$ and $d\ge0$, we have that $y=-CA^{-1}(Bu+Ed)$ at equilibrium. Since $CA^{-1}\le 0$ and $B,E\ge0$, then we have that $y\ge -CA^{-1}Ed$. Therefore, if $d\ge0$ is such that $-CA^{-1}Ed=\mu+\epsilon$, $\epsilon>0$, then output tracking is not achievable. Indeed, for the reference to be reached by the output, we would need the control input to be equal to
\begin{equation}
  u=\dfrac{\mu+CA^{-1}B}{-CA^{-1}E}=\dfrac{-\epsilon}{-CA^{-1}E}<0
\end{equation}
which would violate the nonnegativity of the control input. This leads us to define the following set of admissible disturbance values
\begin{equation}\label{eq:calD}
  \mathcal{D}_\mu:=\left\{d\in\mathbb{R}_{\ge0}:\ \mu+CA^{-1}Ed>0\right\}.
\end{equation}
In particular when $-CA^{-1}E=0$, then the disturbance can be arbitrarily large.

\begin{proposition}
  Let $\mu>0$ be given and assume that $d\in\mathcal{D}_\mu$, then the controlled disturbed system \eqref{eq:mainsystLdist}-\eqref{eq:mainsystK2b} has the following equilibrium point $(x^*,z_1^*,z_2^*)$:
  \begin{equation}
\left(A^{-1}\left(\dfrac{B(\mu+CA^{-1}Ed)}{CA^{-1}B}-Ed\right),-\dfrac{\mu+CA^{-1}Ed}{CA^{-1}Bk},\dfrac{\mu}{\eta k z_1^*}\right).
  \end{equation}
 \end{proposition}

\begin{proposition}
  Let $\mu>0$ be given. For any $E$ of appropriate dimensions, the controlled disturbed system \eqref{eq:mainsystLdist}-\eqref{eq:mainsystK2b} rejects constant disturbances provided that $d\in \mathcal{D}_\mu$ and the corresponding equilibrium point is locally exponentially stable.
\end{proposition}
\begin{proof}
  The controlled disturbed system \eqref{eq:mainsystLdist}-\eqref{eq:mainsystK2b} locally rejects constant disturbances if and only if the matrix
  \begin{equation}
    \begin{bmatrix}
      A & Bk & 0 & \vline & E\\
      0 & -k\eta z_1^* & -k\eta z_2^* & \vline & 0\\
      C & -k\eta z_1^* & -k\eta z_2^* & \vline & 0\\
      \hline
      C & 0 & 0 & \vline & 0
    \end{bmatrix}
  \end{equation}
  is singular. This is equivalent to saying that the DC-gain of the transfer $d\mapsto y$ is zero \blue{or that the (local) transfer function of the closed-loop system has a transmission zero at the zero frequency}.  It is immediate to see that this is the case since the last row is a linear combination of the two previous ones. Note that this would not be the case if the disturbance were acting on the dynamics of the controller or on the output.
\end{proof}

\subsection{Example: gene expression}\label{subsec:ex}

We exemplify here the results of the section on the following gene expression model
\begin{equation}\label{eq:geneexpression}
  \begin{array}{lcl}
    \dot{x}_1(t)&=&-\gamma_1x_1(t)+u(t)\\
    \dot{x}_2(t)&=&k_2x_1(t)-\gamma_2x_2(t)
  \end{array}
\end{equation}
where $x_1,x_2$ and $u$ are the average populations of mRNA, protein and the transcription rate (control input). As in \cite{Briat:15e,Briat:16a,Olsman:19,Olsman:19b}, the idea is to control the system using a positive integral controller. The control input being driven positively to the state $x_1$, the control input can be chosen to be the positive component of the controller and we have
\begin{equation}\label{eq:geneK}
  \begin{array}{lcl}
    \dot{z}_1(t)&=&\mu-\eta kz_1(t)z_2(t)\\
    \dot{z}_2(t)&=&x_2(t)-\eta kz_1(t)z_2(t)\\
    u(t)&=&kz_1(t)
  \end{array}
\end{equation}
where it can be seen that the goal is to have the average number of proteins to track the reference value $\mu>0$. We have the following matrices
\begin{equation}
  A=\begin{bmatrix}
    -\gamma_1 & 0\\
    k_2 & -\gamma_2
  \end{bmatrix},B=\begin{bmatrix}
    1\\0
  \end{bmatrix}\ \textnormal{and }C=\begin{bmatrix}
    0 & 1
  \end{bmatrix}.
\end{equation}
Using Theorem \ref{eq:mainconstructive}, we find that
\begin{equation*}
    \bar{k}_\infty=\dfrac{\gamma_1\gamma_2(\gamma_1+\gamma_2)}{k_2}
\end{equation*}
which means that the unique equilibrium point is locally exponentially stable for all $k\in(0,\bar k_\infty)$, $\mu,\eta>0$. We now use Theorem \ref{eq:mainconstructive2} to find that
\begin{equation}
  \bar\eta_\infty = \dfrac{k_2(\gamma_1+\gamma_2)}{\bar\mu\gamma_1\gamma_2}
\end{equation}
which means that the unique equilibrium point is locally exponentially stable for all $\mu\in(0,\bar\mu)$,  $\eta\in(0,\bar\eta_\infty)$, $k>0$.

For completeness, it seems interesting to obtain the same result using Proposition \ref{prop:k_bar}. We have that
\begin{equation}
  P(s,k)=sD(s)+kN(s)
\end{equation}
where $D(s)=(s+\gamma_1)(s+\gamma_2)$ and $N(s)=k_2$. Substituting $(s,k)$ by $(j\bar\omega,\bar k)$ yields
\begin{equation}
   P(j\bar\omega,\bar k)=-\bar\omega^2(\gamma_1+\gamma_2)+kk_2+j\omega(\gamma_1\gamma_2-\omega^2)=0,
\end{equation}
and, hence,
\begin{equation}
  P_R^0(\bar\omega)=-\bar\omega^2(\gamma_1+\gamma_2),\ P_R^1(\bar\omega)=k_2,\ P_I^0(\bar\omega)=\omega(\gamma_1\gamma_2-\omega^2) \textnormal{ and }P_I^1(\bar\omega)=0.
\end{equation}

We then obtain that
\begin{equation}
  P_R^1(\bar\omega)P^0_I(\bar\omega)-P_R^0(\bar\omega)P_I^1(\bar\omega)=k_2\bar\omega(\gamma_1\gamma_2-\bar\omega^2).
\end{equation}
The only positive root to this equation is given by $\bar\omega=(\gamma_1\gamma_2)^{1/2}$ and we get that the corresponding value for $\bar k$ is given by
\begin{equation}
  \bar k=-\dfrac{P_R^0(\bar\omega)}{P_R^1(\bar\omega)}=-\dfrac{-\bar\omega^2(\gamma_1+\gamma_2)}{k_2}=\dfrac{\gamma_1\gamma_2(\gamma_1+\gamma_2)}{k_2}.
\end{equation}

For numerical purposes, let us consider the parameters $\gamma_1=1$ h$^{-1}$, $\gamma_2=1$ h$^{-1}$, $k_2=1$ h$^{-1}$. In such a case, we must choose  $k<\bar k_\infty=2$. Picking then $k=1/3$ h$^{-1}$, $k\eta=10$ nM$^{-1}$ h$^{-1}$ and $\mu=1$ nM h$^{-1}$, we obtain the simulation results depicted in Figure \ref{fig:signed:genexp} and Figure \ref{fig:signed:genexpeta}. We can observe the convergence of the output to the reference and that a larger $k\eta$ improves the transient performance. In fact, it seems that when $k\eta$ is large enough, the output trajectories converge to the trajectory that would be obtained using a standard integral control law; see Figure~\ref{fig:signed:stdint}. The bifurcation curve in the $(k,\eta)$-plane is depicted in Figure \ref{fig:signed:bif} where the stable region is located below the curve. The vertical line corresponds to the value $\bar k_\infty$ and we can clearly observe that it is an vertical asymptote for the bifurcation curve illustrating that for $k\in(0,\bar k_\infty)$ the closed-loop system is locally exponentially stable regardless the value of $\eta>0$. \blue{Conversely, the horizontal asymptote indicates the value for $\bar\eta_\infty$ which is equal to 2 here. we can see that when $\eta$ is smaller than this value, the closed-loop system is stable for all gains $k>0$. Figure \ref{fig:signed:bif2} depicts the bifurcation surface in the $(k,\eta k)$-plane with corresponds to the bifurcation curve for the slightly differently parameterized controller in \cite{Briat:15e}. This demonstrates the validity of the results for this controller. Figure \ref{fig:signed:RL1} depicts the root locus in the case where $k=1$ as $\eta$ sweeps from 0 to $\infty$. We can clearly see that the system remains stable as the root locus remains confined in the open left half-plane. This is a consequence of the fact that $\mathcal{M}(1)$ is Hurwitz stable. Note also the presence of the asymptote point towards $-\infty$ for the root escaping to infinity. On the other hand,  Figure \ref{fig:signed:RL2} depicts the case where $k=2.5$ showing that the system becomes unstable if $\eta$ is too large because the matrix $\mathcal{M}(2.5)$ is not Hurwitz stable.}
\begin{figure}[!]
  \centering
  \includegraphics[width=0.70\textwidth]{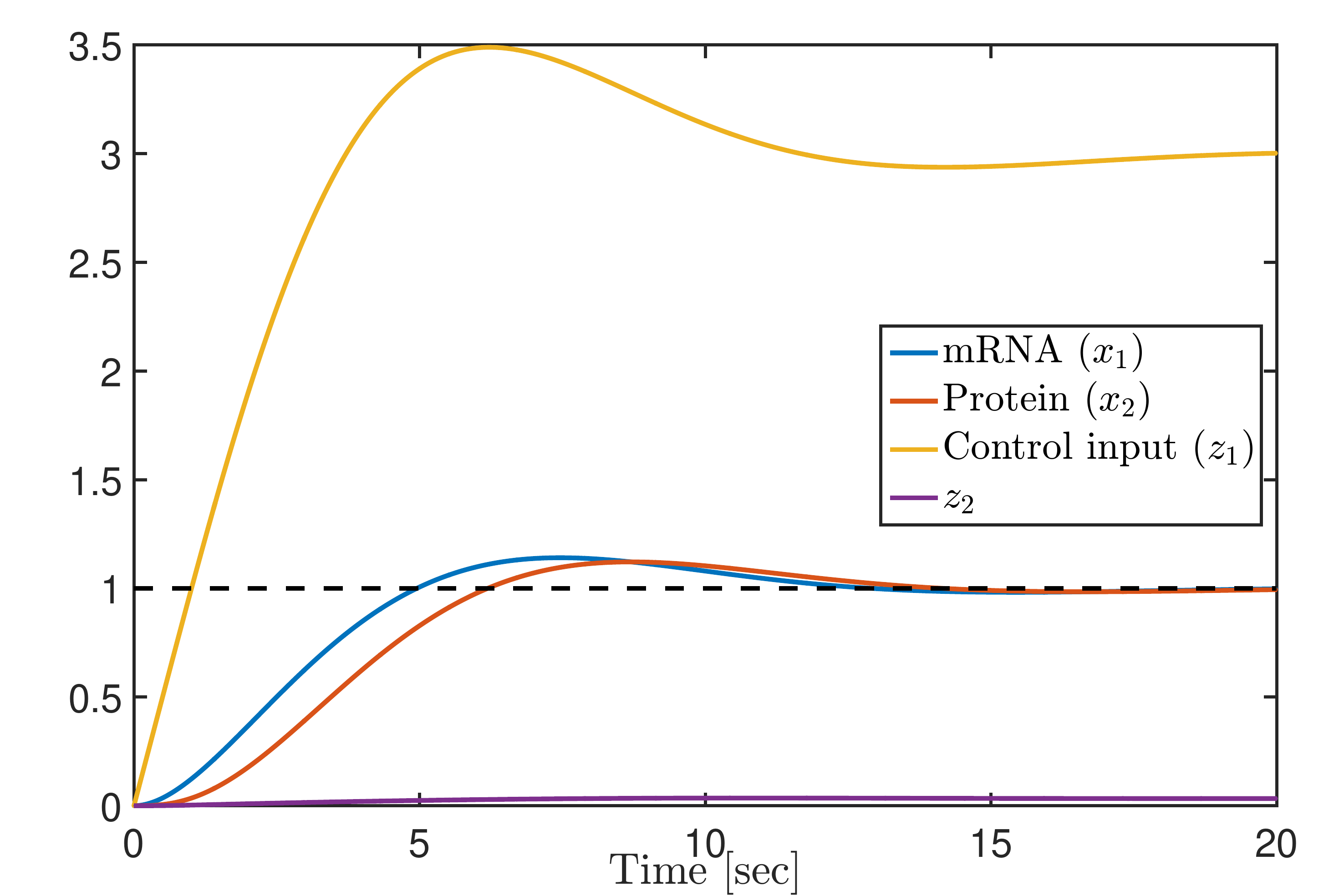}
  \caption{Controlled gene expression network with $k=1/3$, $k\eta=10$ and $\mu=1$.}\label{fig:signed:genexp}
\end{figure}

\begin{figure}[!]
  \centering
  \includegraphics[width=0.7\textwidth]{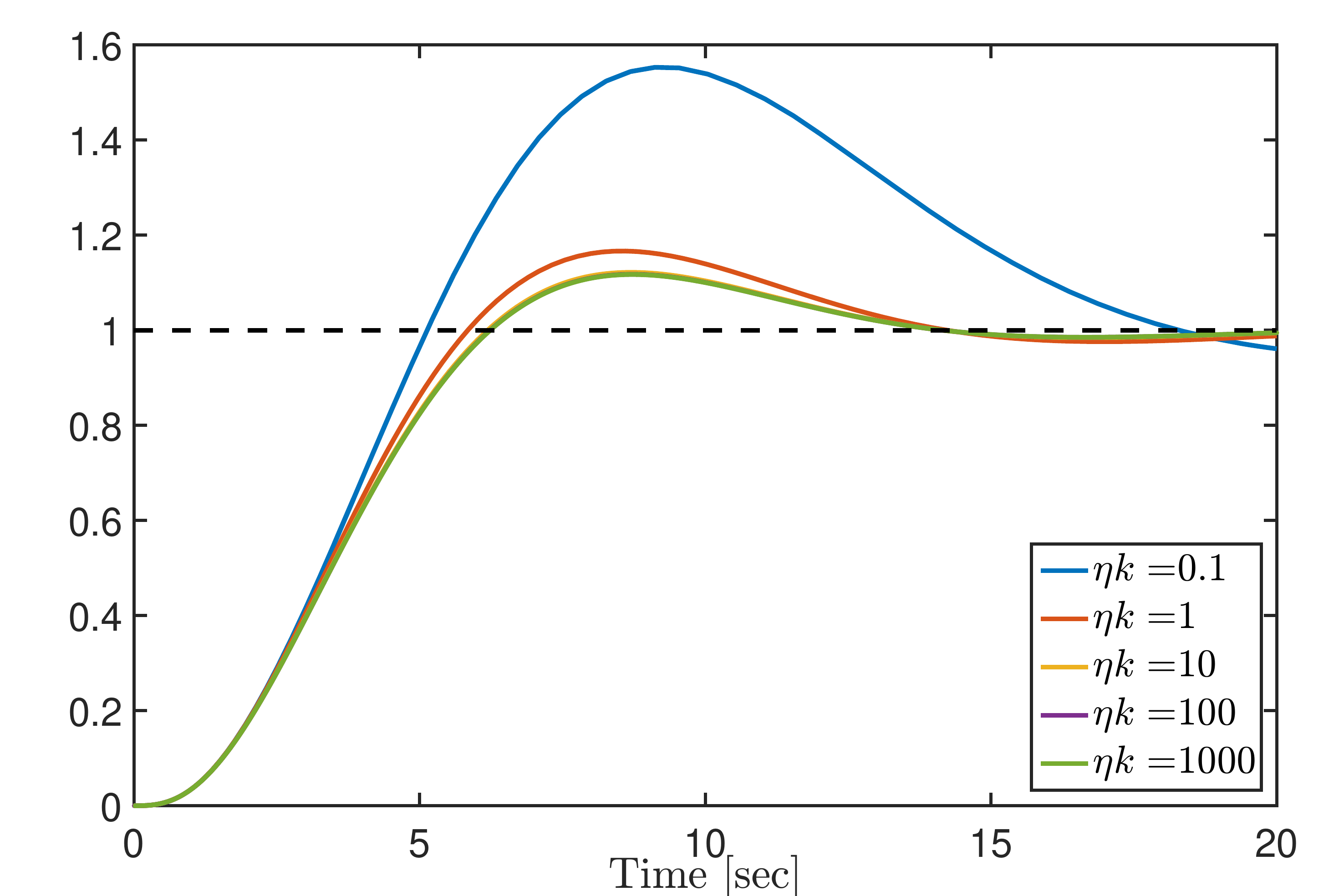}
  \caption{Evolution of the protein concentration with $k=1/3$, $\mu=1$ and different values for $\eta k$.}\label{fig:signed:genexpeta}
\end{figure}

\begin{figure}[!]
  \centering
  \includegraphics[width=0.7\textwidth]{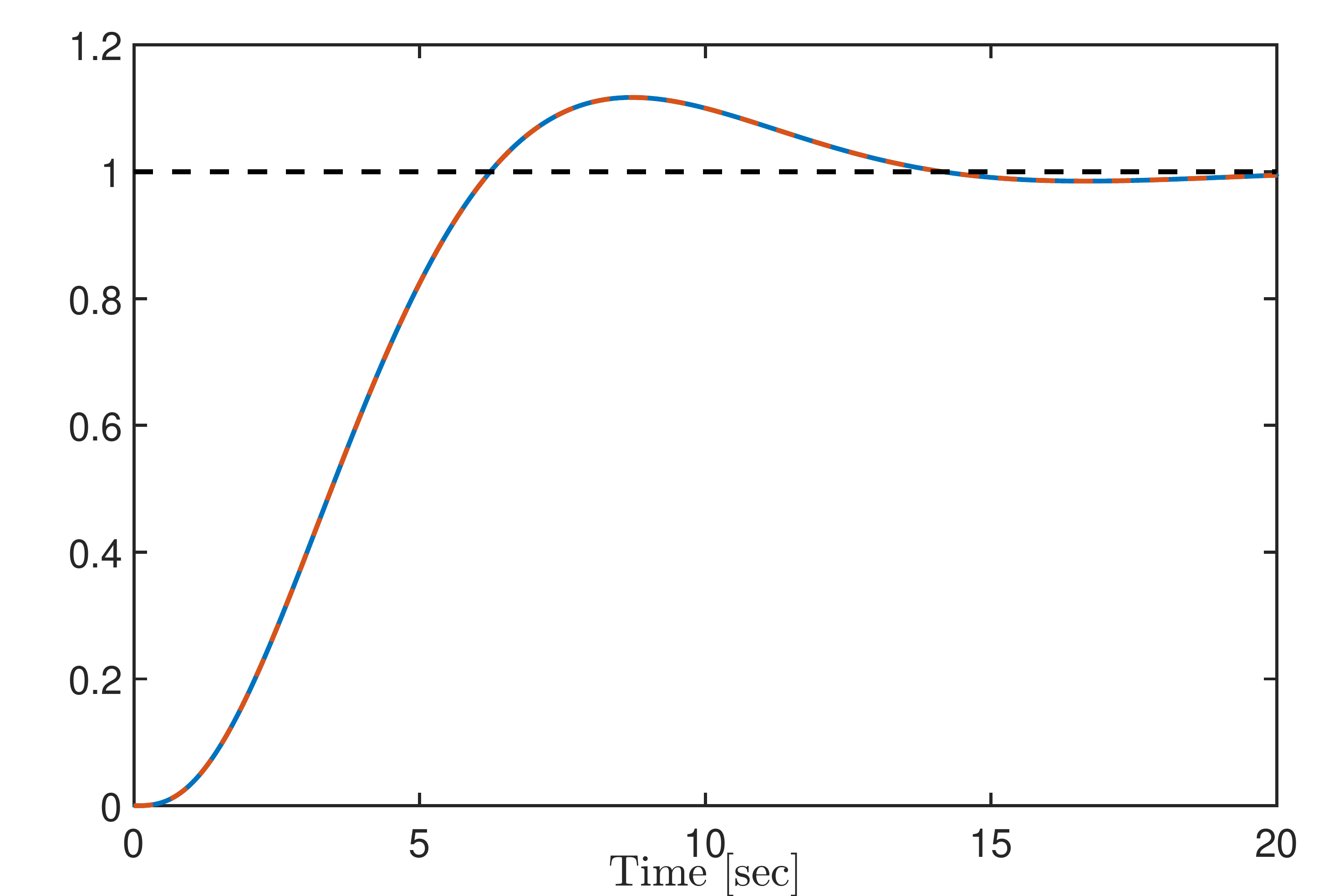}
  \caption{Comparison of the output trajectory for the system controlled with an antithetic integral control with $k\eta=1000$ (blue) and a standard integral control (red). The trajectories are so close that we cannot distinguish them.}\label{fig:signed:stdint}
\end{figure}

\begin{figure}[!]
  \centering
  \includegraphics[width=0.7\textwidth]{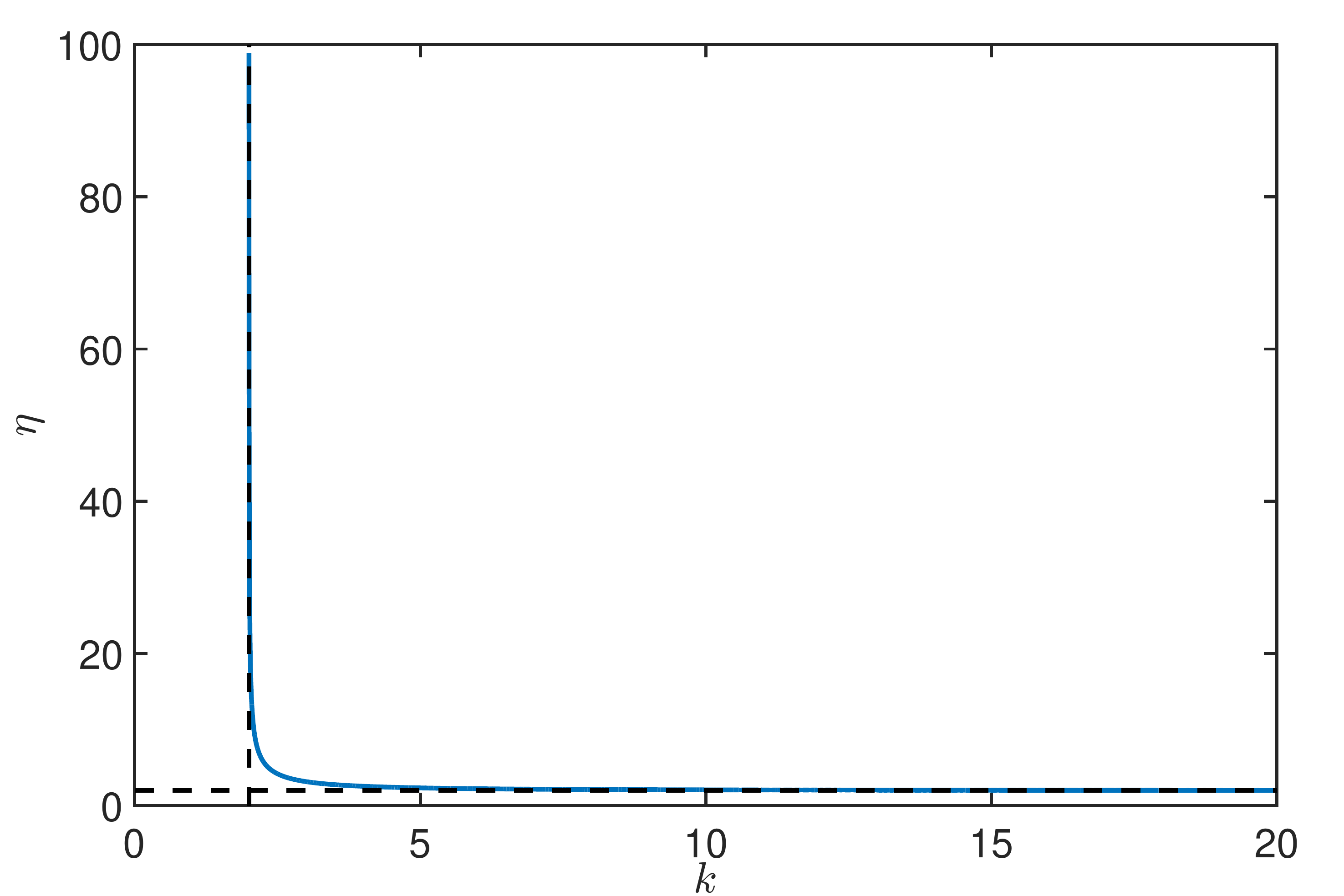}
  \caption{Bifurcation curve in the $(k,\eta)$ plane for the gene expression network with $\gamma_m=1$, $\gamma_p=1$, $k_p=1$ and $\mu=1$. The vertical line corresponds to the value $\bar k_\infty$ whereas the horizontal one corresponds to $\bar\eta_\infty$.}\label{fig:signed:bif}
\end{figure}

\begin{figure}[!]
  \centering
  \includegraphics[width=0.7\textwidth]{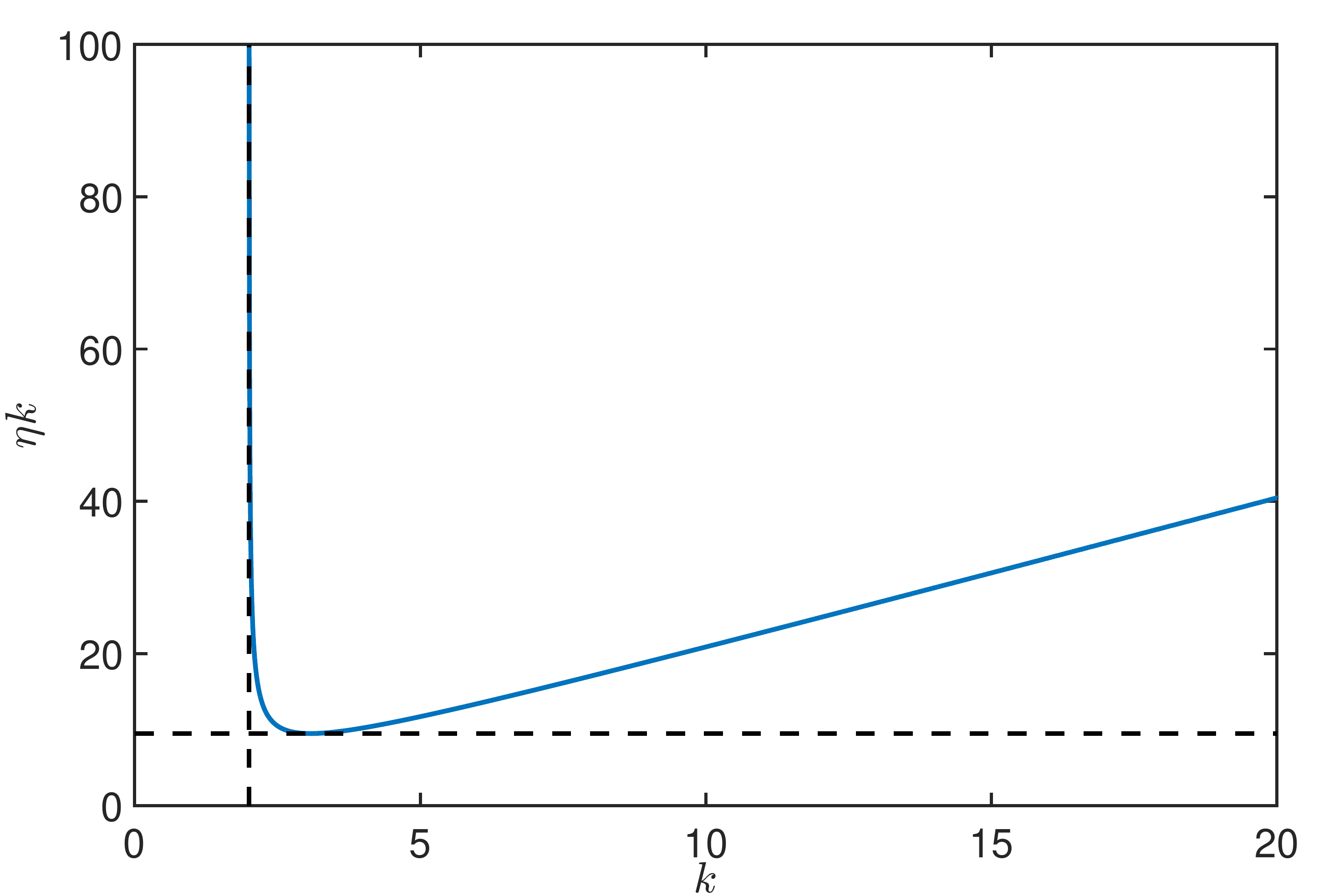}
  \caption{Bifurcation curve in the $(k,k\eta)$ plane for the gene expression network with $\gamma_m=1$, $\gamma_p=1$, $k_p=1$ and $\mu=1$. The vertical line corresponds to the value $\bar k$ whereas the horizontal one corresponds to the maximum value for $\eta k$ for which  the system is stable for all $k>0$. The value is approximately equal to 9.4815 in the current scenario.}\label{fig:signed:bif2}
\end{figure}

\begin{figure}[!]
  \centering
  \includegraphics[width=0.7\textwidth]{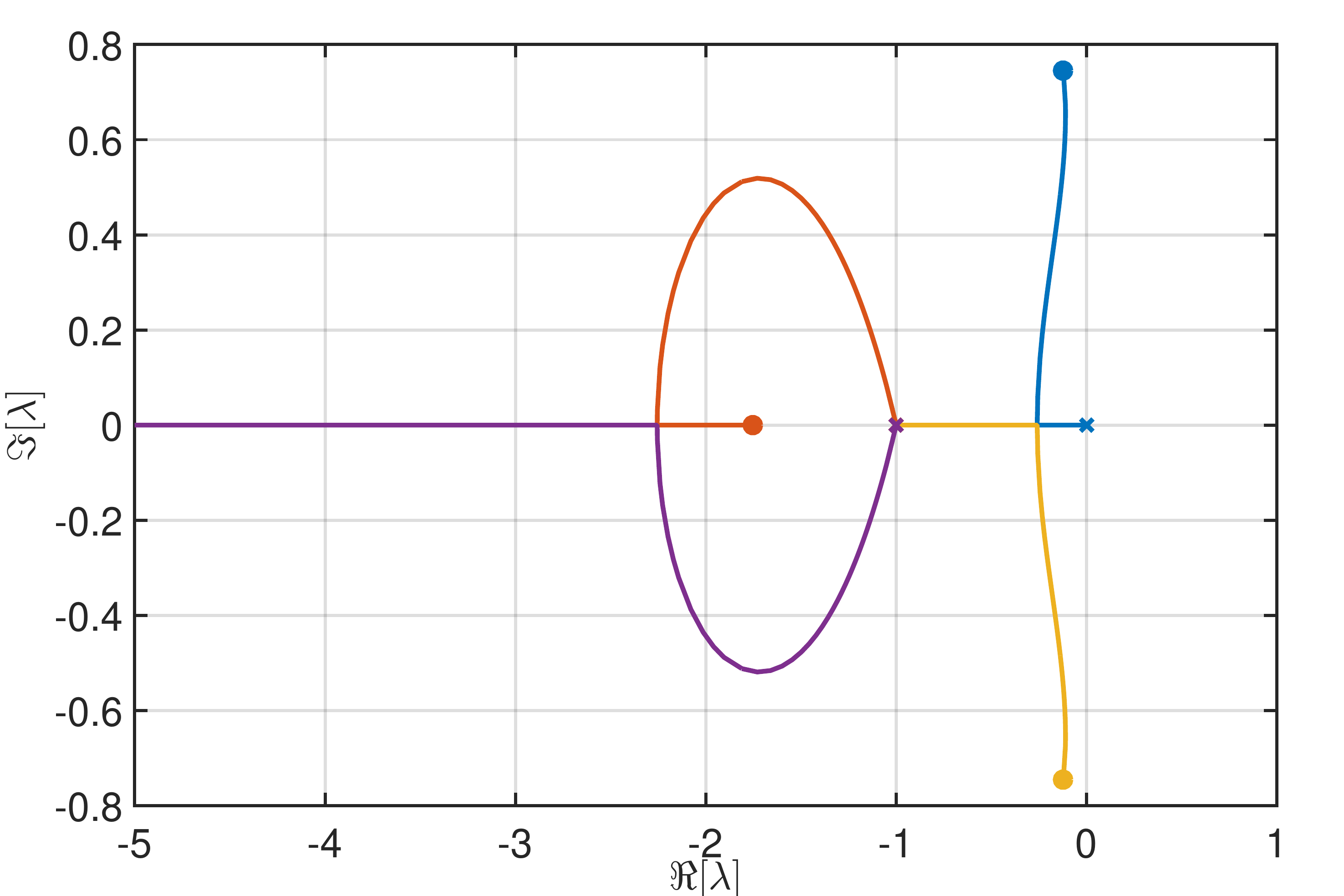}
  \blue{\caption{Root locus of the closed-loop network as $\eta$ increases when $k=1$. The root locus starts from the zeros denoted by crosses and end up at the poles indicated by circles which coincide with the eigenvalues of the matrix $\mathcal{M}(1)$. We can clearly see the asymptote going to $-\infty$ and that the root locus stays in the open left half-plane.}}\label{fig:signed:RL1}
\end{figure}

\begin{figure}[!]
  \centering
  \includegraphics[width=0.7\textwidth]{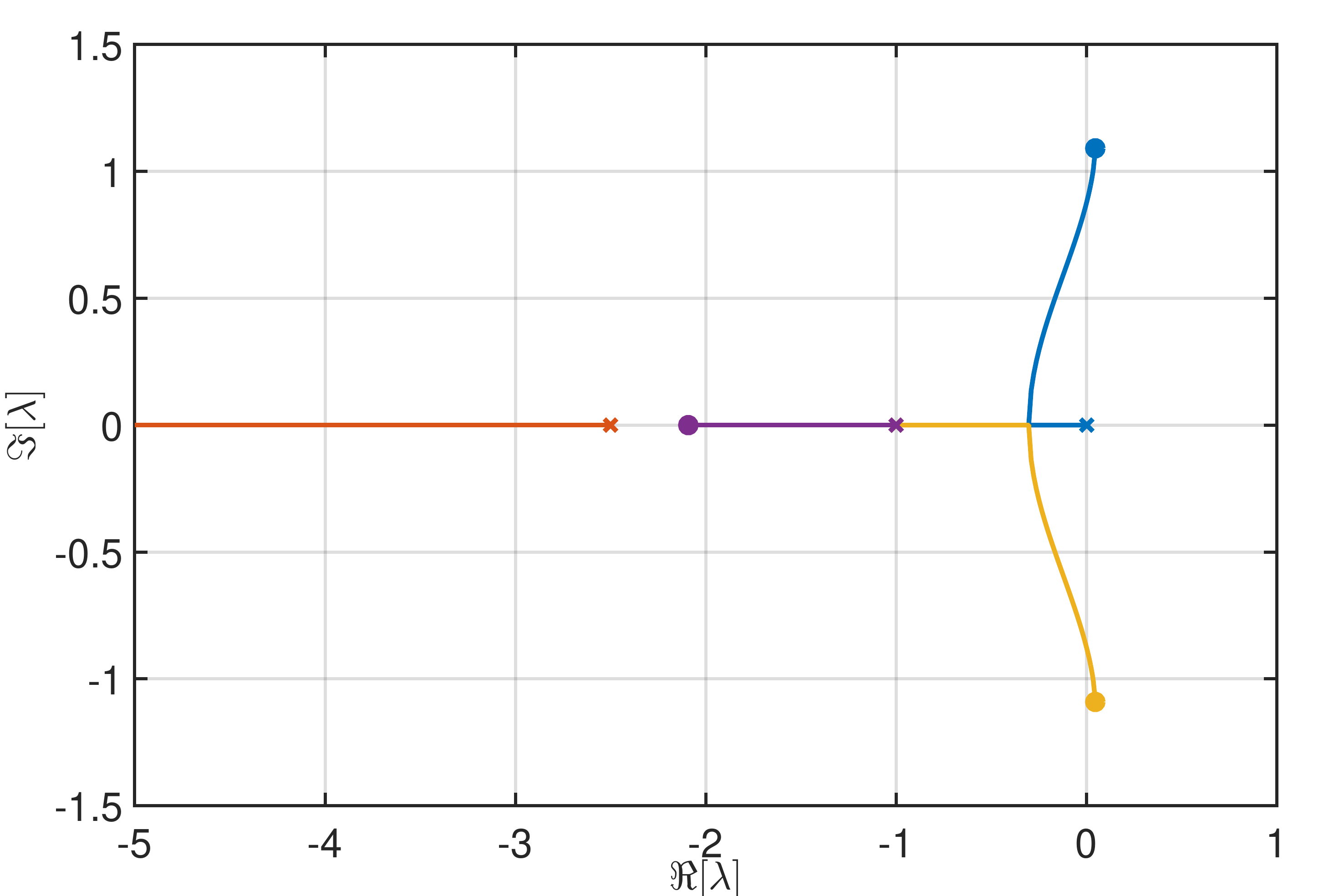}
\blue{\caption{Root locus of the closed-loop network as $\eta$ increases when $k=2.5$. The root locus starts from the zeros denoted by crosses and end up at the poles indicated by circles which coincide with the eigenvalues of the matrix $\mathcal{M}(2.5)$. We can clearly see the asymptote going to $-\infty$ but the root locus does not stay confined in the open left half-plane because the matrix $\mathcal{M}(2.5)$ is not Hurwitz stable.}}\label{fig:signed:RL2}
\end{figure}

\section{Antithetic integral control of nonlinear internally positive systems}\label{sec:AIC_nonlinear}

\blue{This section aims at demonstrating that the theory developed in the previous section for linear positive systems also applies to some classes of nonlinear positive systems. Nonlinear positive systems are briefly recalled in Section \ref{sec:nlp}. The main stability results are presented in Section \ref{sec:nlpstab} and are illustrated in Section \ref{sec:nlpex} through several examples taken from biochemistry and epidemiology.}

\subsection{Nonlinear positive systems}\label{sec:nlp}

Let us consider here the following nonlinear system
\begin{equation}\label{eq:NL}
\begin{array}{rcl}
  \dot{x}(t)&=&f(x(t),u(t)),\ x(0)=x_0\\
  y(t)&=&h(x(t))
\end{array}
\end{equation}
where $x,x_0\in\mathbb{R}^n$, $u\in\mathbb{R}$ and $y\in\mathbb{R}$ are the state of the process, the initial condition, the control input and the measured output, respectively.

\begin{proposition}\label{prop:posNL}
The following statements are equivalent:
\begin{enumerate}[label=({\alph*})]
  \item The system \eqref{eq:NL} is internally positive; i.e. for any $x_0\ge0$ and $u(t)\ge0$, for all $t\ge0$, we have that $x(t)\ge0$ and $y(t)\ge0$, for all $t\ge0$.
  \item The following conditions hold:
  \begin{itemize}
    \item   $f_i(x,u)\ge0$ whenever $x_i=0$ and $x_j,u\ge0$, $j=1,\ldots,N$, $j\ne i$ and for all $i=1,\ldots,n$;
\item   $h(x)\ge0$ for all $x\ge0$.
  \end{itemize}
\end{enumerate}
\end{proposition}
Note that those conditions naturally reduce to those in Proposition \ref{prop:posL} in the linear case.

Define now the following sets
\begin{equation}
  \mathscr{Y}:=\left\{h(x):f(x,u)=0,u\in\mathbb{R}_{\ge0}\right\},
\end{equation}
\begin{equation}
  \mathscr{U}(y):=\left\{u:f(x,u)=0,y=h(x)\right\},
\end{equation}
\begin{equation}
  \mathscr{X}(y):=\left\{x:f(x,u)=0,y=h(x)\right\},
\end{equation}
which characterize the set of equilibrium output values for any given input value, the set of all constant inputs associated with a given equilibrium output value and the set of the equilibrium state values given an equilibrium value for the output. We assume for simplicity that
\begin{assumption}\label{hyp:3}
  For all $y\in\mathscr{Y}$, the sets $\mathscr{U}(y)$ and $\mathscr{X}(y)$ are singletons which are denoted by $u^*(y)$ and $x^*(y)$ their only element, respectively, and the map $F(u):=\left\{h(x):f(x,u)=0\right\}$ is continuous and strictly monotonic.
  %
\end{assumption}
\begin{assumption}\label{hyp:b}
  For all $y\in\mathscr{Y}$, the matrix
  \begin{equation}\label{eq:implicit}
    \left.\dfrac{\partial f}{\partial x}\right|_{(x,u)\in\mathscr{U}(y)\times\mathscr{X}(y)}
  \end{equation}
  is invertible.
\end{assumption}

We are now in position to properly defined the linearized system associated with \eqref{eq:NL} about an equilibrium point uniquely defined by the value of the output at stead-state:
\begin{proposition}
  Let us consider the system \eqref{eq:NL} that is assumed to be internally positive and to satisfy Assumption \ref{hyp:3} and Assumption \ref{hyp:b}. Then, the equilibrium point associated with the steady-state output $y^*=\mu\in\mathscr{Y}$ of the nonlinear system \eqref{eq:NL} is given by
  \begin{equation}
  u^*(\mu)=F^{-1}(\mu),\ x^*(\mu)=g(F^{-1}(\mu))
\end{equation}
where $F^{-1}$ is the inverse function of the nonlinear gain $F$ defined as $F(u)=h(g(u))$ and $g:\mathbb{R}\mapsto\mathbb{R}^n$ is such that $f(g(u),u)=0$. Moreover, the linearized system about that equilibrium is given by
\begin{equation}
  \begin{array}{rcl}
    \dot{\tilde x}(t)&=&\tilde A\tilde x(t)+\tilde B\tilde u(t)\\
    \tilde y(t)&=&\tilde C\tilde x(t)
  \end{array}
\end{equation}
where
\begin{equation}
\begin{bmatrix}
    \tilde A(\mu) & \tilde B(\mu)\\
    \tilde C(\mu) & 0
\end{bmatrix}:=\begin{bmatrix}
  \left.\dfrac{\partial f}{\partial x}\right|_{(x^*(\mu),u^*(\mu))} & \left.\dfrac{\partial f}{\partial u}\right|_{(x^*(\mu),u^*(\mu))}\\
  \left.\dfrac{\partial h}{\partial x}\right|_{(x^*(\mu),u^*(\mu))}& 0
\end{bmatrix}.
\end{equation}
\end{proposition}
\begin{proof}
  Assumption \ref{hyp:3} imposes the existence of a unique steady-state for every possible steady-state output value whereas Assumption \ref{hyp:b} implies that, by virtue of the  implicit function theorem, that the function $g:\mathbb{R}\mapsto\mathbb{R}^n$ mapping steady-state input values to steady-state state values is locally well-defined. Since the function $F$ is monotonic, it is invertible. The result then follows.
\end{proof}

\begin{assumption}\label{hyp:4}
  The matrix $\tilde A$ is Hurwitz stable and $\tilde C(\mu)\tilde A(\mu)^{-1}\tilde B(\mu)\ne0$.
\end{assumption}

As for linear systems, the condition $\tilde C(\mu)\tilde A(\mu)^{-1}\tilde B(\mu)\ne0$ states that the linearized system is locally output controllable. However, the DC-gain may not be positive, the matrix $\tilde A(\mu)$ may not be Metzler and the matrices $\tilde B(\mu),\tilde C(\mu)$ may not be nonnegative. This is, fortunately, not an issue as what really matters is the sign of the DC-gain. Indeed, if the DC-gain is positive then we shall choose $u=kz_1$. Otherwise, we shall choose $u=kz_2$. This will be clarified in the next section.

\subsection{Main results}\label{sec:nlpstab}

\blue{We consider the following version of Problem \ref{prob1}
\begin{problem}\label{prob2}
  Let the reference $\mu>0$ be given and assume that the system  \eqref{eq:NL} is internally positive and that it satisfies Assumption \ref{hyp:3}, Assumption \ref{hyp:b} and Assumption \ref{hyp:4}. Find an integral controller such that
  \begin{enumerate}[label=({\alph*})]
    \item the control input $u$ is nonnegative at all times;
    \item the equilibrium point of interest of the closed-loop system consisting of the system\eqref{eq:NL} and the controller is (locally) asymptotically stable;
    \item the output $y$ asymptotically tracks the reference $\mu>0$; i.e. $y(t)\to\mu$ as $t\to\infty$;
    \item the closed-loop system locally rejects constant disturbances acting on the input and on the state of the system.
    \end{enumerate}
\end{problem}
Again, the antithetic integral controller
\begin{equation}\label{eq:mainsystK2c}
  \begin{array}{lcl}
    \dot{z}_1(t)&=&\mu-\eta kz_1(t)z_2(t)\\
    \dot{z}_2(t)&=&h(x(t))-\eta kz_1(t)z_2(t)\\
    z(0)&=&z_0
  \end{array}
\end{equation}
will be shown to provide a suitable solution to the Problem \ref{prob2}. The main difference with the linear case is that the local gain of internally linear systems is always positive. We have the following preliminary result:
\begin{proposition}
Assume that Assumption \ref{hyp:3}, Assumption \ref{hyp:b} and Assumption \ref{hyp:4} are satisfied for the system \eqref{eq:NL} and assume that this system is internally positive. Then, the equilibrium point of the closed-loop system consisting of the internally positive system \eqref{eq:NL} and the antithetic integral controller \eqref{eq:mainsystK2c} with $u=kz_1$ is given by $\mathcal{X}_{\mu,\eta,k}^{*,+}:=(x^*,z^*)$ where
\begin{equation}\label{eq:point2}
 x^*(\mu)=g(F^{-1}(\mu)),\ z_1^*(\mu)=\dfrac{F^{-1}(\mu)}{k}\ \textnormal{and }z_2^*(\mu)=\dfrac{\mu}{\eta F^{-1}(\mu)} .
\end{equation}
When $u=kz_2$, the equilibrium point is denoted by $\mathcal{X}_{\mu,\eta,k}^{*,-}:=(x^*,z^*)$ where
\begin{equation}\label{eq:point3}
 x^*(\mu)=g(F^{-1}(\mu)),\ z_1^*(\mu)=\dfrac{\mu}{\eta F^{-1}(\mu)}\ \textnormal{and }z_2^*(\mu)=\dfrac{F^{-1}(\mu)}{k} .
\end{equation}
\end{proposition}
It is interesting to


We then have the following result in the case where $u=kz_1$:
\begin{theorem}\label{th:AICNL1}
 Let the conditions in Assumption \ref{hyp:3}, Assumption \ref{hyp:b} and Assumption \ref{hyp:4} be verified for the system \eqref{eq:NL} and assume that this system is internally positive. Let $\mu\in\mathscr{Y}$ be given. Assume further that $F(u)$ is monotonically increasing or, equivalently, that the local gain is positive for all $\mu\in\mathscr{Y}$.
  Then, for any $k\in(0,\bar{k}_\infty^+(\mu))$, the equilibrium point $\mathcal{X}_{\mu,\eta,k}^{*,+}$ is locally asymptotically stable for all $\eta>0$ where
   \begin{equation}
    \bar{k}^+_\infty(\mu):=\sup\left\{\kappa\ge0\ \textnormal{s.t.}\ \begin{bmatrix}
      \tilde A(\mu) & \tilde B(\mu)\kappa\\
      -\tilde C(\mu) & 0
    \end{bmatrix}\ \textnormal{is Hurwitz stable}\right\}.
  \end{equation}
\end{theorem}
\begin{proof}
Let the conditions in Assumption \ref{hyp:3}, Assumption \ref{hyp:b} and Assumption \ref{hyp:4} be verified. Then, the linearized system about the unique equilibrium point induced by $y^*(\mu)=\mu$ is given by
\begin{equation}\label{eq:localNL1}
    \begin{bmatrix}
      \dot{\tilde{x}}(t)\\
      \dot{\tilde{z}}_1(t)\\
      \dot{\tilde{z}}_2(t)
    \end{bmatrix}=\begin{bmatrix}
     \tilde  A(\mu) & \tilde B(\mu)k & 0\\
      0 & -\dfrac{k\mu}{F^{-1}(\mu)} & -\eta F^{-1}(\mu)\\
      \tilde C(\mu) & -\dfrac{k\mu}{F^{-1}(\mu)} & -\eta F^{-1}(\mu)
    \end{bmatrix} \begin{bmatrix}
      \tilde{x}(t)\\
      \tilde{z}_1(t)\\
      \tilde{z}_2(t)
    \end{bmatrix}
  \end{equation}
  where $\tilde{x}=x-x^*(\mu),\tilde{z}_1=z_1-z_1^*(\mu)$ and $\tilde{z}_2=z_2-z_2^*(\mu)$. The rest of the proof follows from exactly the same steps as in the linear case since the matrix is of similar structure. The only difference is that the results may now depend on the set-point $\mu$.
\end{proof}
We can see that the result reduces to that of Theorem \ref{eq:mainconstructive} when $F(u)=-CA^{-1}Bu$ and, hence, $F^{-1}(\mu)=-\mu/CA^{-1}B$. Indeed, substituting this expression in the matrix in \eqref{eq:localNL1} exactly yields the matrix in \eqref{eq:linear}. The value for  $\bar{k}^+_\infty(\mu)$ can be computed using the approaches described in Section \ref{sec:linan3}. In particular, if the transfer function $\tilde C(\mu)(sI-\tilde A(\mu))^{-1}\tilde B(\mu)$ is strictly positive real, then $\bar{k}^+_\infty(\mu)=\infty$.

Let us address now the case where $u=kz_2$:
\begin{theorem}\label{th:AICNL2}
 Let the conditions in Assumption \ref{hyp:3}, Assumption \ref{hyp:b} and Assumption \ref{hyp:4} be verified for the system \eqref{eq:NL} and assume that this system is internally positive. Let $\mu\in\mathscr{Y}$ be given.  Assume further that $F(u)$ is monotonically decreasing or, equivalently, that the local gain is negative for all $\mu\in\mathscr{Y}$.

 Then, for any $k\in(0,\bar{k}_\infty^-(\mu))$, the equilibrium point $\mathcal{X}_{\mu,\eta,k}^{*,+}$ is locally asymptotically stable for all $\eta>0$ where
  \begin{equation}\label{eq:jdsjdsk}
    \bar{k}^-_\infty(\mu):=\sup\left\{\kappa\ge0\ \textnormal{s.t.}\ \begin{bmatrix}
      \tilde A(\mu) & \tilde B(\mu)\kappa\\
      \tilde C(\mu) & 0
    \end{bmatrix}\ \textnormal{is Hurwitz stable}\right\}.
  \end{equation}
\end{theorem}
\begin{proof}
  Let the conditions in Assumption \ref{hyp:3}, Assumption \ref{hyp:b} and Assumption \ref{hyp:4} be verified. Then, the linearized system about the unique equilibrium point induced by $y^*(\mu)=\mu$ is given by
\begin{equation}\label{eq:localNL2}
    \begin{bmatrix}
      \dot{\tilde{x}}(t)\\
      \dot{\tilde{z}}_1(t)\\
      \dot{\tilde{z}}_2(t)
    \end{bmatrix}=\begin{bmatrix}
     \tilde  A(\mu) & 0 & \tilde B(\mu)k \\
      0 & -\eta F^{-1}(\mu) & -\dfrac{k\mu}{F^{-1}(\mu)}\\
      \tilde C(\mu)  & -\eta F^{-1}(\mu) & -\dfrac{k\mu}{F^{-1}(\mu)}
    \end{bmatrix} \begin{bmatrix}
      \tilde{x}(t)\\
      \tilde{z}_1(t)\\
      \tilde{z}_2(t)
    \end{bmatrix}
  \end{equation}
  where $\tilde{x}=x-x^*(\mu),\tilde{z}_1=z_1-z_1^*(\mu)$ and $\tilde{z}_2=z_2-z_2^*(\mu)$. The rest of the proof follows from the same arguments as in the linear case with the difference that the matrix has a different structure. The proof is only sketched for brevity. We first show that a small enough $k>0$ makes the above matrix Hurwitz stable. Indeed, the zero-eigenvalue of the matrix when $k=0$ moves in the direction $-\tilde C(\mu)\tilde A(\mu)^{-1}\tilde B(\mu)$ for some sufficiently small $k>0$. Since, the local gain is negative, then the zero eigenvalue moves in the open left half-plane, making the matrix Hurwitz stable for some sufficiently small $k>0$. Similarly, we show that for any sufficiently small $\eta>0$ makes the matrix Hurwitz stable. The zero-eigenvalue of the matrix when $\eta=0$ moves in the direction $-\tilde C(\mu)\tilde A(\mu)^{-1}\tilde B(\mu)F^{-1}(\mu)<0$  for some sufficiently small $\eta>0$. Hence the zero eigenvalue moves in the open left half-plane, making the matrix Hurwitz stable for some sufficiently small $k>0$. Analogously, we can show that the matrix Hurwitz stable for any sufficiently large $\eta>0$ provided that the matrix \eqref{eq:jdsjdsk} is Hurwitz stable. To prove that the Hurwitz stability of the matrix in \eqref{eq:jdsjdsk} is a necessary and sufficient condition for the local stability of the equilibrium point $\mathcal{X}_{\mu,\eta,k}^{*,-}$, we can invoke a root locus argument which develops exactly as in the linear case. The details are omitted.
\end{proof}
The value for  $\bar{k}^-_\infty(\mu)$ can be computed using the approaches described in Section \ref{sec:linan3}. In particular, if the transfer function $-\tilde C(\mu)(sI-\tilde A(\mu))^{-1}\tilde B(\mu)$ (not the minus sign) is strictly positive real, then $\bar{k}^-_\infty(\mu)=\infty$.}

\subsection{Examples}\label{sec:nlpex}

\noindent\textbf{SIS model.} Let us consider the following deterministic SIS model
\begin{equation}\label{eq:SIS}
  \begin{array}{lcl}
    \dot{x}_1(t)&=&-\beta x_1(t)x_2(t)+\alpha x_2(t)\\
    \dot{x}_2(t)&=&\beta x_1(t)x_2(t)-\alpha x_2(t)
  \end{array}
\end{equation}
where $x_1(t)$ and $x_2(t)$ represent the susceptible and infectious people, respectively. Since the system verifies the conservation law $x_1(t)+x_1(t)=N$ for all $t\ge0$, it can therefore be reduced to the system one-dimensional system:
\begin{equation}\label{eq:SIS2}
  \dot{x}_1(t)=-\beta x_1(t)(N-x_1(t))+\alpha(N-x_1(t)).
\end{equation}
Assume now that we would like to control the number of susceptible people to a certain value $\mu<N$, and that it is possible to control the recovery rate $\alpha$, thus we set $\alpha(t)=u(t)$. Hence,
\begin{equation}
  \mathscr{Y}=[0,\mu],\mathscr{X}(\mu)=\{\mu\} \textnormal{ and } \mathscr{U}(\mu)=F^{-1}(\mu)=\{\mu\beta\}
\end{equation}
and $F(u)=u/\beta$. Hence, the system satisfies Assumption \ref{hyp:3}. Since, the nonlinear gain is an monotonically increasing function of $u$, then the local DC-gain is positive and, therefore, we choose the controller
\begin{equation}\label{eq:SISK}
  \begin{array}{lcl}
    \dot{z}_1(t)&=&\mu-\eta kz_1(t)z_2(t)\\
    \dot{z}_2(t)&=&x_1(t)-\eta kz_1(t)z_2(t)\\
    u(t)&=&kz_1(t)
  \end{array}
\end{equation}
where $k,\eta>0$ are the controller parameters.

We then have the following result:
\blue{\begin{theorem}
The unique equilibrium point
\begin{equation}\label{eq:SISeq}
  \begin{array}{lclclclclcl}
    x_1^*&=&\mu, &&   z_1^*&=&\dfrac{\beta\mu}{k},&& z_2^*&=&\dfrac{1}{\beta\eta}
  \end{array}
\end{equation}
of the closed-loop system \eqref{eq:SIS2}-\eqref{eq:SISK} is locally exponentially stable for all $N,k,\eta,\beta>0$ and all $\mu\in(0,N)$.
\end{theorem}
\begin{proof}
We need first to show that the system satisfies the assumptions. We have already shown that Assumption \ref{hyp:3} holds. The linearized system around the equilibrium point \eqref{eq:SISeq} is given by
\begin{equation}\label{eq:dsldsdjsld2}
  \dot{X}(t)=\begin{bmatrix}
    \beta(\mu-N) & k(N-\mu) & 0\\
    0 & -k/\beta & -\eta\beta\mu\\
    1 & -k/\beta & -\eta\beta\mu
  \end{bmatrix}X(t).
\end{equation}
The matrix $\tilde A(\mu)=\beta(\mu-N)$ is negative since $\mu<N$ and we have that $-C(\mu)A(\mu)^{-1}B(\mu)=1/\beta$. Therefore, the assumptions are satisfied. The result then follows from the fact that the local transfer function of the system given by
\begin{equation}
  H(s)=\dfrac{N-\mu}{s+\beta(N-\mu)}
\end{equation}
is strictly positive real for all $\beta,N$ and $0<\mu<N$.
%
\end{proof}}
%



For simulation purposes, let us consider the controlled SIS-model  \eqref{eq:SIS2}-\eqref{eq:SISK} with the parameters $\beta=1$, $\eta k=13$, $k=2$, $N=100$ and $\mu=99$; i.e. we aim at maintaining 99\% of the population healthy. The initial condition is set to $x_1(0)=90$ and $z_1(0)=z_2(0)=0$. We obtain the simulation results depicted in Figure \ref{fig:SIS}.
\begin{figure}[H]
  \centering
  \includegraphics[width=0.7\textwidth]{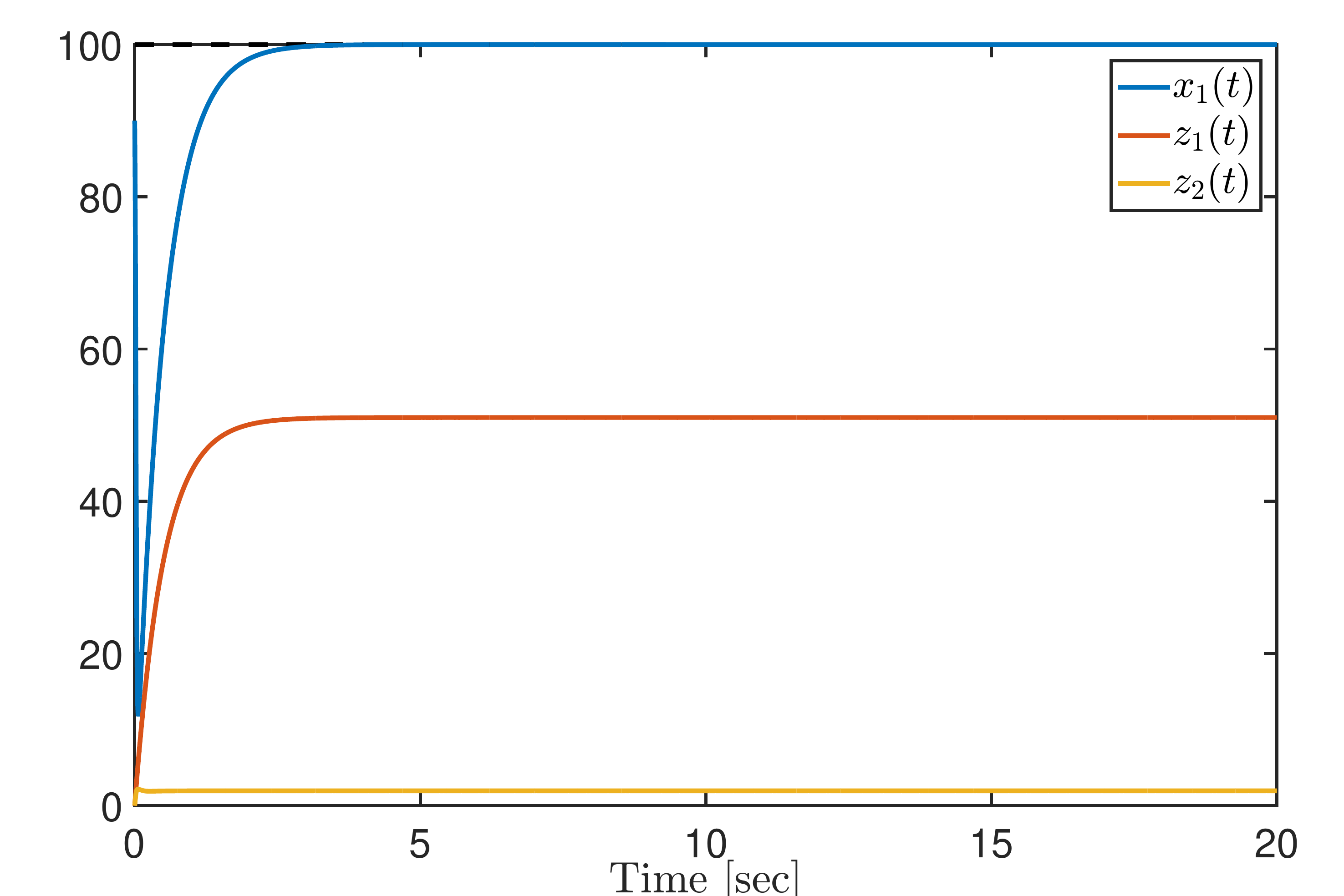}
  \caption{Controlled SIS system \eqref{eq:SIS2}-\eqref{eq:SISK} with the parameters $\beta=1$, $\eta k=13$, $k=2$, $N=100$ and $\mu=99$.}\label{fig:SIS}
\end{figure}

%
\blue{\noindent\textbf{Gene expression with controlled translation.} Let us consider the following gene expression model with repressed translation:

\begin{equation}\label{eq:trans}
  \begin{array}{lcl}
    \dot{x}_1&=&-\gamma_1 x_1+k_1\\
    \dot{x}_2&=&\dfrac{k_2x_1}{1+u}-\gamma_2x_2
  \end{array}
\end{equation}
where $x_1$ and $x_2$ denote the concentration of the mRNA and its associated protein, respectively. Choosing $y=x_2$, we have that
\begin{equation}
  F(u)=\dfrac{k_1k_2}{\gamma_1\gamma_2(1+u)}\ \textnormal{and, hence, }F^{-1}(\mu)=\dfrac{k_1k_2}{\gamma_1\gamma_2\mu}-1.
\end{equation}
Hence,
\begin{equation}
  \mathscr{Y}=\left[0,\dfrac{k_1k_2}{\gamma_1\gamma_2}\right],\mathscr{X}(\mu)=\left\{\begin{bmatrix}
    \dfrac{k_1}{\gamma_1}\\
    \mu
  \end{bmatrix}\right\}\textnormal{ and }\mathscr{U}(\mu)=\left\{\dfrac{k_1k_2}{\gamma_1\gamma_2\mu}-1\right\}.
\end{equation}
As a result, the system satisfies Assumption \ref{hyp:3}. Moreover, since the function $F(u)$ is a monotonically decreasing function of $u$, then the local gain is negative and, therefore, we need to consider the controller
\begin{equation}\label{eq:transK}
  \begin{array}{lcl}
    \dot{z}_1(t)&=&\mu-\eta kz_1(t)z_2(t)\\
    \dot{z}_2(t)&=&x_2(t)-\eta kz_1(t)z_2(t)\\
    u(t)&=&kz_2(t).
  \end{array}
\end{equation}
This yields the following result:
\begin{proposition}
  The unique equilibrium point
\begin{equation}
(x_1^*,x_2^*,z_1^*,z_2^*)=\left(\dfrac{k_1}{\gamma_1},\mu, \dfrac{\gamma_1\gamma_2\mu^2}{\eta(k_1k_2-\gamma_1\gamma_2\mu)},\dfrac{1}{k}\left(\dfrac{k_1k_2}{\gamma_1\gamma_2\mu}-1\right)\right)
\end{equation}
of the closed-loop system \eqref{eq:trans}-\eqref{eq:transK} is locally exponentially stable for all $\gamma_1,\gamma_2,k_1,\eta,k>0$ and all $\mu\in\left(0,\dfrac{k_1k_2}{\gamma_1\gamma_2}\right)$.
\end{proposition}
\begin{proof}
  The linearized closed-loop system is given by
\begin{equation}
  \tilde{A}(\mu)=\begin{bmatrix}
    -\gamma _1 & 0\\
    \dfrac{\gamma_1\gamma_2\mu}{k_1} & -\gamma_2
  \end{bmatrix},\tilde{B}(\mu)=\begin{bmatrix}
    0\\
    \dfrac{-\gamma_1\gamma_2^2}{k_1k_2}
  \end{bmatrix},\tilde{C}(\mu)=\begin{bmatrix}
      0 & 1
    \end{bmatrix}.
\end{equation}
We can see that the matrix $ \tilde{A}(\mu)$ is Hurwitz stable (hence invertible) for all $\mu\in\mathscr{Y}$ and that $$-\tilde{C}(\mu)\tilde{A}(\mu)^{-1}\tilde{B}(\mu)=-\dfrac{\gamma_1\gamma_2\mu^2}{k_1k_2}<0.$$ Hence, the linearized system satisfies Assumption \ref{hyp:4} and Assumption \ref{hyp:b} for all $\mu\in\mathscr{Y}$.  Therefore, we can apply Theorem \ref{th:AICNL2} and consider the transfer function
\begin{equation}
  H_\mu(s):=-\tilde C(\mu)(sI-\tilde A(\mu))^{-1}\tilde B(\mu)=\dfrac{\gamma_1\gamma_2^2}{k_1k_2(s+\gamma_2)}.
\end{equation}
This transfer function as a positive DC-gain and is of first-order type, which implies that it is strictly positive real. This proves the result.
\end{proof}}

\section{Exponential integral control of linear positive systems}\label{sec:exponential_linear}

\blue{We consider in this section the exponential integral controller. First the model of the exponential integral controller is given. The local stability of the different equilibrium points is then analyzed. Finally, a way to compute the maximum value of the exponential rate is provided. It is notably shown that the gain of the controller does not alter the stability properties of the closed-loop system.}

\subsection{The model}

The exponential integral controller is based on the consideration of
\begin{equation}
  \varphi(x)=e^{\alpha x},\alpha>0
\end{equation}
as regularizing function. This leads to the following result:
\begin{proposition}
 The exponential integral controller is given by
  \begin{equation}\label{eq:ic2}
  \begin{array}{rcl}
        \dot{v}(t)&=&\alpha v(t)(\mu-y(t)),v(0)=v_0\\
        u(t)&=&kv(t)
  \end{array}
  \end{equation}
  where $k>0$ is the gain of the controller and $\alpha>0$ is the exponential rate of the controller of the controller. The closed-form solution of the above system is given by $$v(t)=v_0\exp\left(\alpha\int_0^t[\mu-y(s)]ds\right).$$
\end{proposition}

\blue{The name comes from the fact that the integral action is exponentiated. Interestingly,  this controller has a form that is reminiscent of the logistic model; see e.g. \cite{Verhulst:38,Brauer:12}.} It also shares some connection with biochemical models exhibiting the property of Absolute Concentration Robustness (ACR) \cite{Shinar:10}. Those networks have species whose stationary values that do not depend on initial conditions. The connection between Absolute Concentration Robustness and integral control has been recently clarified in \cite{Cappelletti:19}.

\begin{remark}
  The controller \eqref{eq:ic2} is only valid when the local DC-gain of the open-loop system is positive, which is always the case for linear internally positive systems. In the case of nonlinear internally positive systems, the local DC-gain can be negative and, in such a case, one should use the controller
  $$\dot{v}(t)=\alpha v(t)(y(t)-\mu)$$
  where we have swapped the error terms.
\end{remark}

\subsection{Main results}

We now address the question of the existence of equilibrium points for the closed-loop system \eqref{eq:mainsystL}-\eqref{eq:ic2}.
\begin{proposition}
  The closed-loop system \blue{consisting of the system \eqref{eq:mainsystL} and the controller \eqref{eq:ic2}} admits two equilibrium points:
  \begin{enumerate}[label=({\alph*})]
    \item the first one is the trivial equilibrium point
    \begin{equation}\label{eq:trivial_eqpt}
      (x^*,v^*)=(0,0)
    \end{equation} whereas
  \item the second one is the positive equilibrium point
  \begin{equation}\label{eq:positive_eqpt}
     (x^*,z^*)=\left(\dfrac{A^{-1}B\mu}{CA^{-1}B},\dfrac{-\mu}{CA^{-1}Bk}\right)
  \end{equation}
  which exists provided that $\mu>0$.
  \end{enumerate}
\end{proposition}
Clearly, the first equilibrium is undesirable since it does solve the control problem stated in Problem \ref{prob1}. Indeed, the output does not track the desired set-point $\mu$ unlike in the second case. The results below show that the first equilibrium point is structurally unstable whereas the second one can be made locally exponentially stable through a suitable choice for the controller parameters. Let us start first with the trivial equilibrium point:
\begin{proposition}\label{prop:stab_zero_1}
\blue{Let us consider the closed-loop system consisting of the system \eqref{eq:mainsystL} and the controller \eqref{eq:ic2}.} Then, the zero-equilibrium point \eqref{eq:trivial_eqpt} is unstable for all $\alpha,k,\mu>0$.
\end{proposition}
\begin{proof}
The linearized system about that equilibrium point is given by
  \begin{equation}
  \begin{bmatrix}
    \dot{\tilde{x}}(t)\\
    \dot{\tilde{v}}(t)
  \end{bmatrix}=\begin{bmatrix}
    A & Bk\\
    0 & \alpha \mu
  \end{bmatrix} \begin{bmatrix}
    \tilde{x}(t)\\
    \tilde{v}(t)
  \end{bmatrix}
\end{equation}
and is obviously unstable since $\alpha\mu>0$.
\end{proof}
While the zero equilibrium point is unstable in the deterministic setting, it actually becomes an absorbing state in the stochastic setting and, therefore, cannot be used in that setting unlike the antithetic integral controller.

We address now the stability of the positive equilibrium point:
\begin{proposition}\label{prop:stab_positive_1}
   \blue{Let us consider the closed-loop system consisting of the system \eqref{eq:mainsystL} and the controller \eqref{eq:ic2} and let} $\mu>0$ be given. Then, there exists an $\bar{\alpha}>0$ such that the positive equilibrium point \eqref{eq:positive_eqpt}
%
  is locally exponentially stable for all $k>0$ and all $\alpha\in(0,\bar{\alpha})$ where
    \begin{equation}\label{eq:baralpha1}
      \bar{\alpha}:=\sup\left\{\nu>0:\ \begin{bmatrix}
        A & B\\
          \dfrac{\nu\mu C}{CA^{-1}B} & 0
      \end{bmatrix}\ \textnormal{Hurwitz}\right\}.
    \end{equation}
\end{proposition}
\begin{proof}
The linearized system about that equilibrium point is given by
  \begin{equation}\label{eq:locstab1}
  \begin{bmatrix}
    \dot{\tilde{x}}(t)\\
    \dot{\tilde{v}}(t)
  \end{bmatrix}=\begin{bmatrix}
    A & Bk\\
    \dfrac{\alpha C\mu}{CA^{-1}Bk} & 0
  \end{bmatrix} \begin{bmatrix}
    \tilde{x}(t)\\
    \tilde{v}(t)
  \end{bmatrix}.
\end{equation}
It is immediate to see that the eigenvalues of the above matrix do not depend on $k$. Therefore, we only need to characterize the dependence of the eigenvalues on $\alpha,\mu>0$. We prove now using a perturbation argument that the matrix is Hurwitz for some sufficiently small $\alpha>0$. When $\alpha=0$, the resulting matrix has $n$ stable eigenvalues (those of $A$) and one eigenvalue at 0. The eigenvalue at 0 has normalized left- and right-eigenvectors given by
\begin{equation}
\ u_\ell=\begin{bmatrix}
    0_{1\times n} & 1
  \end{bmatrix},\  u_r=\begin{bmatrix}
    -A^{-1}B \\ 1
  \end{bmatrix}
\end{equation}
respectively. The theory of perturbation of eigenvalues says that the zero-eigenvalue locally changes under the effect of the perturbation of $\alpha$ around $\alpha=0$ according to the relation $\lambda(\alpha)=\lambda(0)+u_\ell Mu_r\alpha+o(\alpha)$ where $\lambda(0)=0$ and
\begin{equation*}
  M = \begin{bmatrix}
    0 & 0\\
    \dfrac{C\mu}{CA^{-1}B} & 0
  \end{bmatrix}.
\end{equation*}
We therefore obtain that $\lambda(\alpha)=-\mu \alpha+o(\alpha)$. This means that by slightly perturbing $\alpha$ from 0 to positive values, we can shift the marginally stable eigenvalue into the open left-half plane. This means that there exists $\bar{\alpha}>0$ such that for all $\alpha<\bar{\alpha}$, the Jacobian matrix is Hurwitz. 
\end{proof}

\blue{We prove below a robustness result which seems to be specific to the integral controller
\begin{theorem}\label{th:robustgain}
  Assume that the system \eqref{eq:mainsystL} is internally positive and that it satisfies Assumption \ref{hyp:1}. Assumer further that the system \eqref{eq:mainsystL} depends on some parameters $p=(p_1,\ldots,p_K)$ and let $\bar{p}_1=(p_1,\ldots,p_\ell)$ and $\bar{p}_2=(p_{\ell+1},\ldots,p_K)$ for some $0<\ell\le K$. Assume further that the transfer function of the system writes
  \begin{equation}
    C(p)(sI-A(p))^{-1}B(p)=G(\bar{p}_1)\tilde H(s,\bar{p}_2)
  \end{equation}
  for some function $G:\mathbb{R}^\ell_{\ge0}\mapsto\mathbb{R}$ and some transfer function $\tilde H(s,\bar{p}_2)$.

  Then, the stability of the equilibrium point does not depend on the parameters in $\bar{p}_1$.
\end{theorem}
\begin{proof}
  The matrix in \eqref{eq:locstab1} is Hurwitz stable if and only if the roots of its characteristic polynomial
  \begin{equation}
    P(s,\alpha):=\det\begin{bmatrix}
    sI-A(p) & -B(p)\\
    -\dfrac{\alpha\mu C(p)}{C(p)A(p)^{-1}B(p)} & s
  \end{bmatrix}=0
  \end{equation}
  are located in the open left half-plane. By virtue of the determinant formula, we have
  \begin{equation}
    P(s,\alpha)=\det(sI-A(p))\det\left(s+\alpha\mu H_n(s,\bar{p}_2)\right)
  \end{equation}
  where $H_n(s)$ is defined as
  \begin{equation}
  H_n(s,\bar{p}_2):=-\dfrac{C(p)(sI-A(p))^{-1}B(p)}{C(p)A(p)^{-1}B(p)}=\dfrac{G(\bar{p}_1)\tilde H(s,\bar{p}_2)}{G(\bar{p}_1)\tilde H(0,\bar{p}_2)}=\dfrac{\tilde H(s,\bar{p}_2)}{\tilde H(0,\bar{p}_2)}
  \end{equation}
  and is independent $\bar{p}_1$. Therefore, the stability is independent of those parameters.
\end{proof}

The above result generalizes the fact that the stability of the positive equilibrium point does not depend on the gain of the controller. In this regard, this controller allows arbitrarily large gain for the system and may be useful for controlling highly sensitive systems. This will be illustrated in the examples.}

\subsection{Computing ${\overline{\alpha}}$}

The following result proposes a numerical method for computing $\bar \alpha$:
\begin{proposition}\label{prop:alpha_bar}
\blue{Let us consider the closed-loop system consisting of the internally positive system \eqref{eq:mainsystL} and the controller \eqref{eq:ic2}. Assume further that the system \eqref{eq:mainsystL} satisfies Assumption \ref{hyp:1} and define the real polynomials $N_R(\omega)$, $D_R(\omega)$, $N_I(\omega)$ and $D_I(\omega)$ as
\begin{equation}
  H_n(j\omega):=-\dfrac{C(sI-A)^{-1}B}{CA^{-1}B}=:\dfrac{N_R(\omega)+jN_I(\omega)}{D_R(\omega)+jD_I(\omega)}
\end{equation}
 and let
 \begin{equation}\label{eq:polycrtic2}
   \Omega:=\left\{\omega>0: Q(\omega):=N_I(\omega)D_I(\omega)+N_R(\omega)D_R(\omega)=0\right\}.
\end{equation}
Then, the equilibrium point \eqref{eq:positive_eqpt} is locally exponentially stable for all $k>0$ and all $\alpha\in(0,\bar\alpha_\infty)$ where
  \begin{equation}
  \bar{\alpha}_\infty:=\inf_{\omega\in\Omega}\dfrac{1}{\mu}\left\{\begin{array}{lcl}
    \dfrac{D_I(\omega)\omega}{N_R(\omega)}&\textnormal{if}&N_R(\omega)\ne0\\
    \dfrac{-D_R(\omega)\omega}{N_I(\omega)}&\textnormal{if}&N_I(\omega)\ne0.
  \end{array}\right.
\end{equation}
When the set $\Omega$ is empty, then $\bar{\alpha}_\infty=\infty$.}
\end{proposition}
\begin{proof}
The matrix in \eqref{eq:locstab1} is Hurwitz stable if and only if the roots of its characteristic polynomial
  \begin{equation}
    P(s,\alpha):=\det\begin{bmatrix}
    sI-A & -B\\
    -\dfrac{\alpha\mu C}{CA^{-1}B} & s
  \end{bmatrix}=0
  \end{equation}
  are located in the open left half-plane. By virtue of the determinant formula, we have
  \begin{equation}\label{eq:P}
    P(s,\alpha)=\det(sI-A)\det\left(s+\alpha\mu H_n(s)\right)=sD(s)+\alpha\mu N_n(s)=0
  \end{equation}
  where $H_n(s)=:N_n(s)/D(s)$ is defined in Proposition \ref{prop:alpha_bar}. Therefore, the above polynomial is Hurwitz stable if and only if $ P(s,\alpha)$ does not have zeros in the closed right-half plane. We use the same approach as for proving Proposition \ref{prop:k_bar} and we view \eqref{eq:P} as a multivariate polynomial. We then look for pairs $(\omega,\alpha)\in\mathbb{R}^2_{>0}$ such that $P(j\omega,\alpha)=0$. Such a pair exists if and only if $\Re[P(j\omega,\alpha)]=0$ and $\Im[P(j\omega,\alpha)]=0$. Expanding those expressions and combining them in matrix form yields
  \begin{equation}
    \begin{bmatrix}
      -\omega D_I(\omega) & N_R(\omega)\\
      \omega D_R(\omega) & N_I(\omega)
    \end{bmatrix}\begin{bmatrix}
      1\\
    \alpha\mu
    \end{bmatrix}=0.
  \end{equation}
  Such a vector exists if and only if the matrix is singular and, therefore, if and only if $\omega Q(\omega)=0$. Since, we do not consider the zero frequency which corresponds to the case $\bar \alpha=0$, we are just left with the condition that $Q(\omega)=0$. The result then follows.
\end{proof}
Unlike the antithetic integral controller, the stability conditions depend on the desired set-point and one cannot make the equilibrium stable for any arbitrarily large set-point for a given $\alpha$. However, the stability of the equilibrium will not depend on the gain $k$ as well as all the parameters exclusively appearing in the DC-gain of the system. This is a consequence of the fact that the stability condition depends on the normalized transfer function. This will be illustrated in the examples.

\subsection{Disturbance rejection/Perfect adaptation}

We characterize here the disturbance rejection properties of the exponential controller. Let us start with the analysis of the zero equilibrium point:
\begin{proposition}
  Assume that $d\in\mathcal{D}_\mu$, then the equilibrium point
  \begin{equation}
    \left(x^*,v^*\right)=\left(-A^{-1}Ed,0\right)
  \end{equation}
  is unstable for all $\alpha,k,\mu>0$.
\end{proposition}
\begin{proof}
The proof is identical to the one of Proposition \ref{prop:stab_zero_1}.
\end{proof}

The following result states stability conditions for the positive equilibrium point:
\begin{proposition}
  Let $\mu>0$ be given and assume that $d\in\mathcal{D}_\mu$, then the equilibrium point
  \begin{equation}
(x^*,v^*)=\left(A^{-1}\left(\dfrac{B(\mu+CA^{-1}Ed)}{CA^{-1}B}-Ed\right),-\dfrac{\mu+CA^{-1}Ed}{CA^{-1}Bk}\right)
  \end{equation}
  is locally exponentially stable for all $k>0$ and any $\alpha\in(0,\bar{\alpha}_d)$ where
  \begin{equation}
     \bar{\alpha}^d_\infty=\dfrac{\mu}{\mu+CA^{-1}Ed}\bar{\alpha}_\infty\ge\bar{\alpha}_\infty
  \end{equation}
  where $\bar{\alpha}_\infty$ is defined in Proposition \ref{prop:alpha_bar}. Moreover, if $\alpha<\bar{\alpha}_\infty$, then the equilibrium point is locally exponentially stable for all $d\in\mathcal{D}_\mu$ and all $k>0$.
\end{proposition}
\begin{proof}
 As we have seen before, the equilibrium value of the state does not change the local linear system. Only the equilibrium value of the state of the integrator matters. Noting that $v^*(d)=v^*(0)+\Delta_vd$ where
 \begin{equation}
  \begin{array}{rcl}
    v^*(0)&=&\dfrac{-\mu}{CA^{-1}Bk}>0\\
    \Delta_v&=&-\dfrac{CA^{-1}E}{CA^{-1}Bk}\le0
  \end{array}
\end{equation}
allows us to conclude that $v^*(d)\le v^*(0)$ for all $d\in\mathcal{D}_\mu$. In this case, the state matrix of the linearized system becomes
\begin{equation}
  \begin{bmatrix}
    A & Bk\\
    -\alpha v^*(d)C & 0
  \end{bmatrix}
\end{equation}
where we can observe that  $\alpha v^*(d)C$ is a decreasing function of $d$ and, therefore, the presence of the disturbance tends to decrease the loop gain. From the disturbance-free stability result, we can see that it is enough to rescale the value for $\bar \alpha_\infty$ to get the new bound, This completes the proof.
\end{proof}

\subsection{Examples}

\blue{\textbf{Gene expression.} We consider again the gene expression system from Section \ref{subsec:ex}. In this case, we have that
\begin{equation}
  H_n(s)=\dfrac{\gamma_1\gamma_2}{(s+\gamma_1)(s+\gamma_2)}.
\end{equation}
Since the system has a relative degree equal to two, then $\bar\alpha_\infty$ is finite. The polynomial $Q(\omega)$ is given for this system by
  \begin{equation}
    Q(\omega)=\gamma_1\gamma_2(\omega^2-\gamma_1\gamma_2)
  \end{equation}
  and, therefore, the only positive root is $\sqrt{\gamma_1\gamma_2}$. We can verify that
  \begin{equation}
    H_n(j\sqrt{\gamma_1\gamma_2})=\dfrac{-j\sqrt{\gamma_1\gamma_2}}{\gamma_1+\gamma_2}
  \end{equation}
  and is located on the imaginary axis. We then get that
  \begin{equation}
    \bar\alpha_\infty=\dfrac{\gamma_1+\gamma_2}{\mu}.
  \end{equation}
  We can see that the stability condition does not depend on $k_2$ as expected from Theorem \ref{th:robustgain}. As a result, the equilibrium point will be stable for all $k_2>0$, this a clear difference with the antithetic integral controller which requires that $k<\bar{k}_\infty=\gamma_1\gamma_2(\gamma_1+\gamma_2)/k_2$ is the strong binding regime. Clearly, if $k_2$ is much larger than $\gamma_1\gamma_2(\gamma_1+\gamma_2)$, the stability of the system will not be robust with respect to the controller gain. The exponential controller allows one to overcome this situation. Minor drawback of the integral controller is that the exponential rate $\alpha$ needs to be tuned with an a priori maximum value for the reference and its nonlinear behavior.

  For numerical purposes, let us consider the parameters $\gamma_1=1$ h$^{-1}$, $\gamma_2=1$ h$^{-1}$, $k_2=1$ h$^{-1}$ as well as $k=1$. We obtain the results depicted in Figure \ref{fig:gene_exp_alpha_y} and Figure \ref{fig:gene_exp_alpha_v} where we can see the influence of $\alpha$ on the dynamics of the closed-loop system. The robustness with respect to the gain of the system is illustrated for various values for $k_2$ is illustrated in Figure \ref{fig:gene_exp_k2_y}.}

  %
%

\begin{figure}[H]
  \centering
  \includegraphics[width=0.7\textwidth]{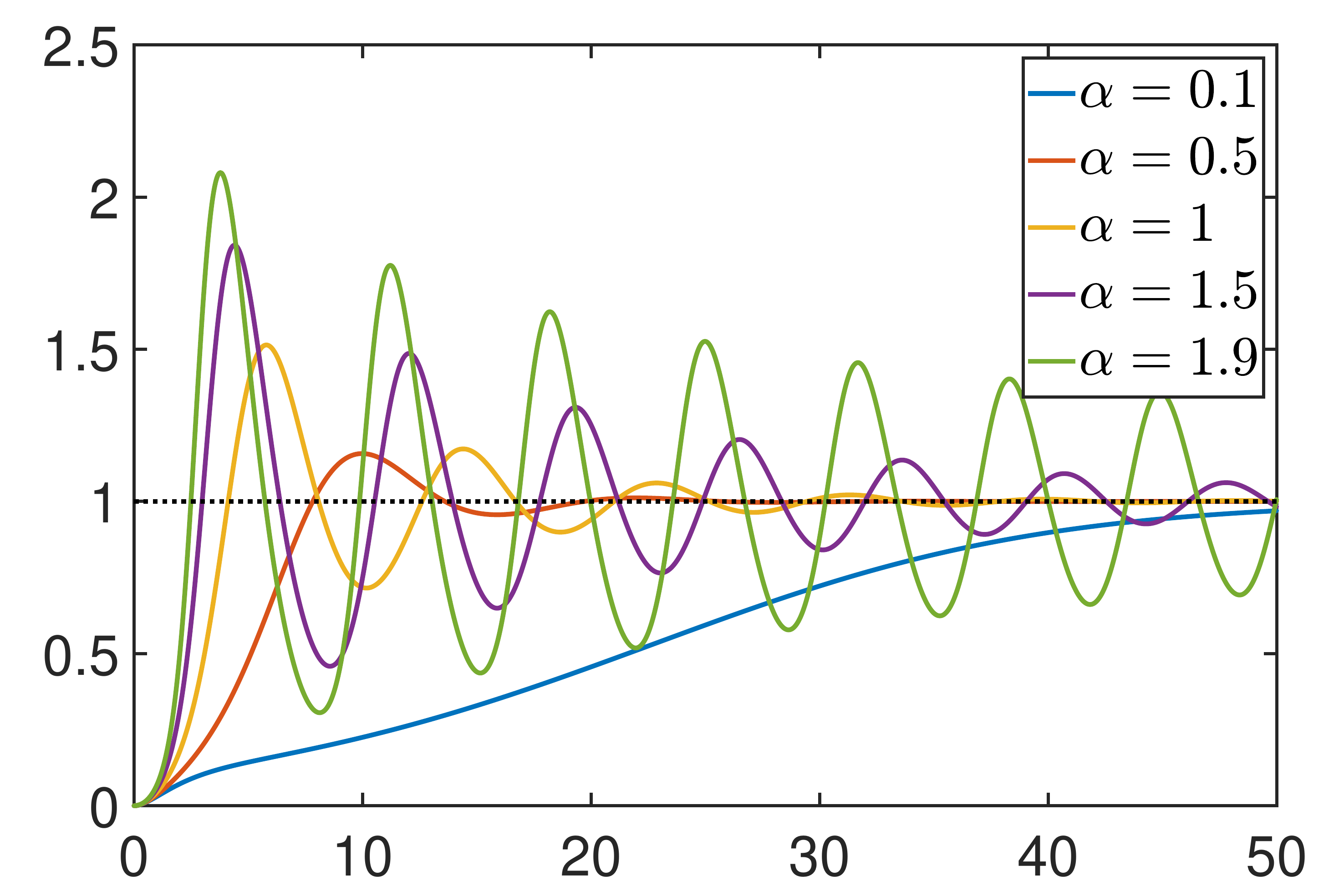}
  \caption{Output of the controlled gene expression system consisting of the system \eqref{eq:geneexpression} and the exponential controller \eqref{eq:ic2} for various values for $\alpha$.}\label{fig:gene_exp_alpha_y}
\end{figure}

\begin{figure}[H]
  \centering
  \includegraphics[width=0.7\textwidth]{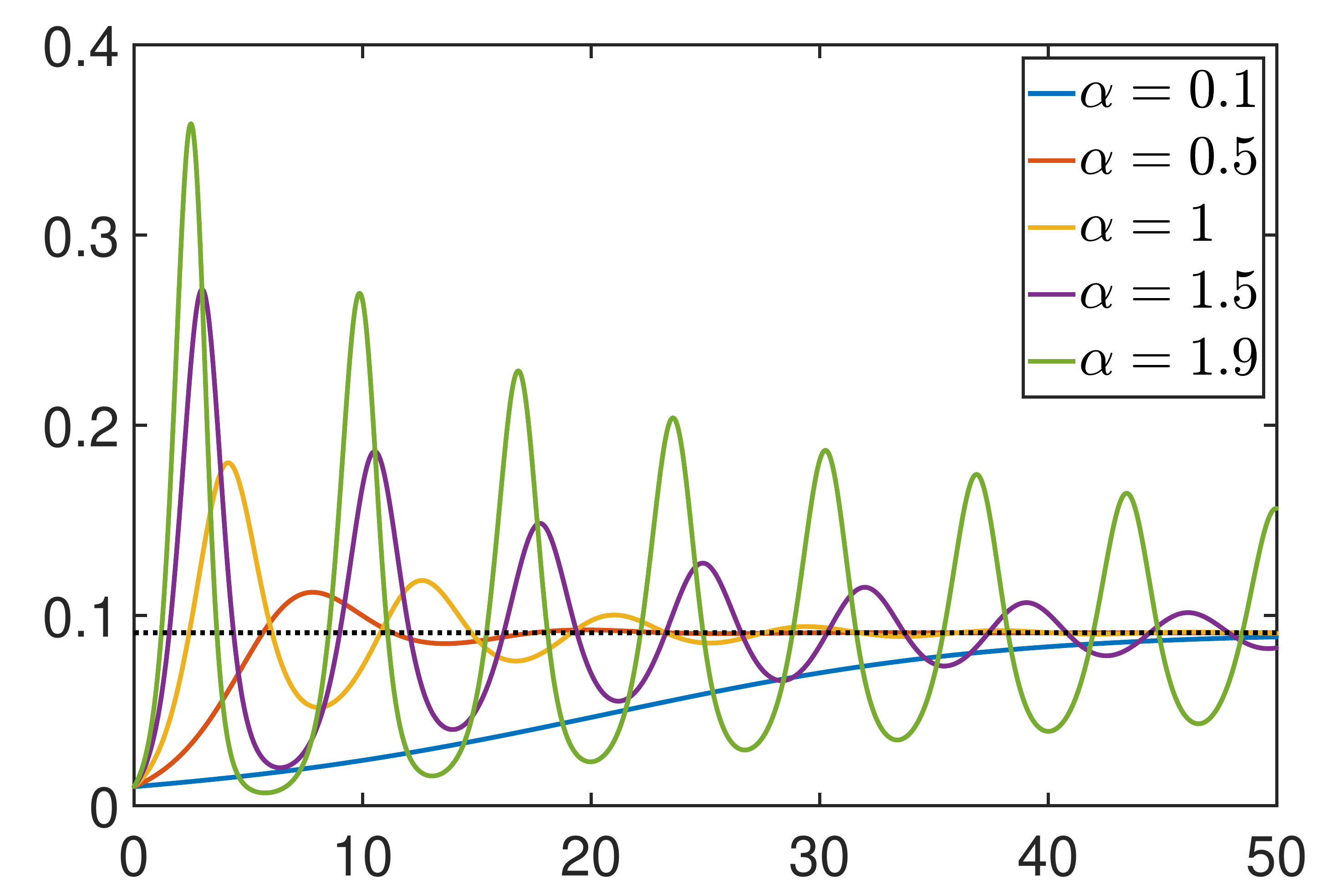}
  \caption{State of the exponential controller \eqref{eq:ic2} in closed loop with the gene expression system consisting of the system \eqref{eq:geneexpression} for various values for $\alpha$.}\label{fig:gene_exp_alpha_v}
\end{figure}

\begin{figure}[H]
  \centering
  \includegraphics[width=0.7\textwidth]{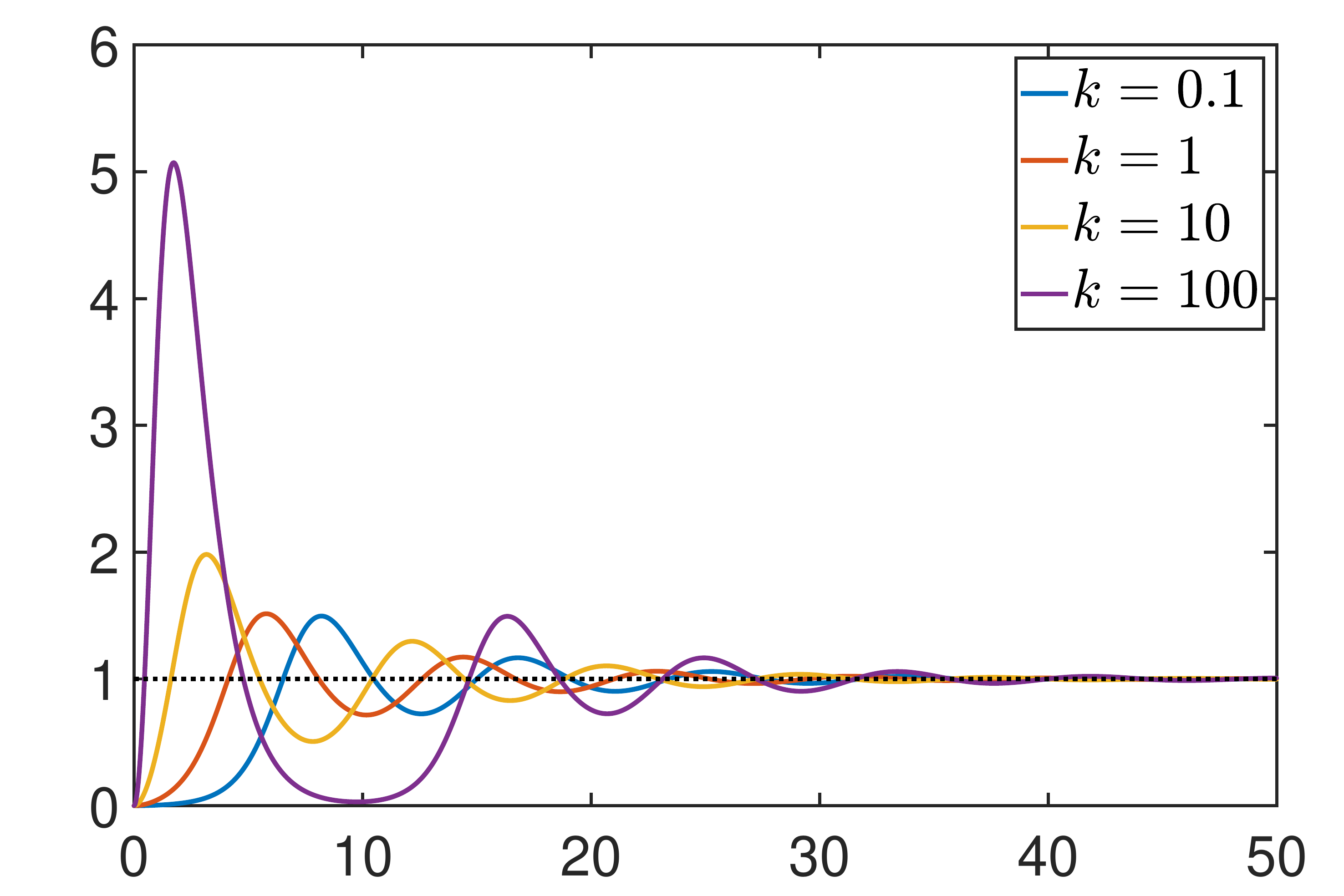}
  \caption{Output of the controlled gene expression system consisting of the system \eqref{eq:geneexpression} and the exponential controller \eqref{eq:ic2} for various values for $k_2$ and for $k=1$, $\alpha=0.5$.}\label{fig:gene_exp_k2_y}
\end{figure}

\textbf{Gene expression with protein maturation.} Let us consider the gene expression system with protein maturation
  \begin{equation}\label{eq:example_th}
    \begin{array}{lcl}
      \dot{x}_1(t)&=&k_1-\gamma_1x_1(t)\\
      \dot{x}_2(t)&=&k_2x_1(t)-(\gamma_2+k_3)x_2(t)\\
      \dot{x}_3(t)&=&k_3x_2(t)-\gamma_3x_3(t).
    \end{array}
  \end{equation}
where $x_1,x_2$ and $x_3$ denote the concentration of mRNA, protein and maturated protein molecules, respectively. We choose the transcription rate as the control input and the controlled output is the concentration of maturated proteins. This makes the system linear, stable and positive. Its normalized transfer function $H_n(s)$ is given by
\begin{equation}
  H_n(s)=\dfrac{\gamma_1\gamma_3(\gamma_2+k_3)}{(s+\gamma_1)(s+\gamma_2+k_3)(s+\gamma_3)}
\end{equation}
which is readily seen to be independent of $k_2$. Since the above transfer function has relative degree 3, then $\bar\alpha_\infty$ is finite. The associated polynomial $Q(\omega)$ is given by
  \begin{equation}
    Q(\omega)=\gamma_1\gamma_3(\gamma_2+k_3)-\omega^2(\gamma_1 + \gamma_2 + \gamma_3 + k_3)
  \end{equation}
  from which we can conclude that $Q(\bar\omega)=0$ with $\bar\omega=\sqrt{\dfrac{\gamma_1\gamma_3(\gamma_2+k_3)}{\gamma_1 + \gamma_2 + \gamma_3 + k_3}}$. This finally yields the bound
  \begin{equation}
    \bar{\alpha}_\infty=\dfrac{(\gamma_1 + \gamma_3)(\gamma_1 + \gamma_2 + k_3)(\gamma_2 + \gamma_3 + k_3)}{\mu(\gamma_1 + \gamma_2 + \gamma_3 + \gamma_3)^2}.
  \end{equation}

\section{Logistic integral control of linear positive systems}\label{sec:logistic_linear}

\blue{We consider here a saturated version of the exponential integral controller, namely the logistic integral controller. First the model of the logistic integral controller is given. The local stability of the different equilibrium points is then analyzed. Finally, a way to compute the maximum value of the exponential rate is provided.}

\subsection{The model}

The model of the logistic integral controller is based on the use of the logistic function
\begin{equation}
  \varphi(x)=\frac{\beta }{1+e^{-\alpha x}},\alpha,\beta>0
\end{equation}
as regularizing function. This leads to the following result:
\begin{proposition}
  The logistic integral controller is given by
  \begin{equation}\label{eq:ic3}
  \begin{array}{lcl}
        \dot{v}(t)&=&\dfrac{\alpha}{\beta}v(t)(\beta-v(t))(\mu-y(t))\\
        u(t)&=&kv(t)
  \end{array}
  \end{equation}
  where $k>0$ is the gain of the controller, $\alpha>0$ is the exponential rate of the controller and $\beta>0$ is the \emph{saturation bound} of the controller. The closed-form solution of the above system is given by $$\displaystyle v(t)=\dfrac{\beta}{\displaystyle 1+v_0\exp\left(\alpha\int_0^t[\mu-y(s)]ds\right)}.$$
\end{proposition}

This model can be viewed as an extension of the exponential model and allows to capture for the saturation at the value $\beta$. This is the reason why we get the additional term $\beta-v(t)$ in the controller dynamical expression. . It is important to note here that this controller is only valid when the local DC-gain of the open-loop system is positive. When it is negative, the exponential integral controller then takes the form $$\dfrac{\alpha}{\beta}v(t)(\beta-v(t))(y(t)-\mu).$$


\subsection{Local stability analysis}

The use of a logistic integral controller yields a closed-loop system with three equilibrium points:
\begin{proposition}
  The closed-loop system \blue{consisting of the system \eqref{eq:mainsystL} and the controller \eqref{eq:ic3}} admits three equilibrium points:
  \begin{enumerate}[label=({\alph*})]
    \item the first one is the zero equilibrium point
    \begin{equation}\label{eq:trivial_eqpt_rat}
      (x^*,v^*)=(0,0)
    \end{equation}
    \item the second one is the saturating equilibrium point
    \begin{equation}\label{eq:saturating_eqpt_rat}
        (x^*,z^*)=\left(-A^{-1}Bk\beta,\beta\right)
    \end{equation}
  \item the third one is the positive equilibrium point
  \begin{equation}\label{eq:positive_eqpt_rat}
     (x^*,z^*)=\left(\dfrac{A^{-1}B\mu}{CA^{-1}B},\dfrac{-\mu}{CA^{-1}Bk}\right)
  \end{equation}
  which exists provided that $0<\mu\le-CA^{-1}Bk\beta$.
  \end{enumerate}
\end{proposition}

Similarly to as in the previous section, the zero and saturating equilibrium points are undesirable equilibrium points as they do not achieve tracking for the constant set-point for the controlled output. We prove below that the zero equilibrium point is unstable:
\begin{proposition}
 \blue{Let us consider the closed-loop system consisting of the internally positive system \eqref{eq:mainsystL} and the controller \eqref{eq:ic3}}. Assume further that the system \eqref{eq:mainsystL} satisfies Assumption \ref{hyp:1} and that $\mu\in(0,-CA^{-1}Bk\beta)$. Then the zero-equilibrium point \eqref{eq:trivial_eqpt_rat} is unstable for all $\alpha,\beta,k>0$.
\end{proposition}
\begin{proof}
The linearized system about that equilibrium point is given by
\begin{equation}
  \begin{bmatrix}
    \dot{\tilde{x}}(t)\\
    \dot{\tilde{v}}(t)
  \end{bmatrix}=\begin{bmatrix}
    A & Bk\\
    0 & \alpha\mu
  \end{bmatrix} \begin{bmatrix}
    \tilde{x}(t)\\
    \tilde{z}(t)
  \end{bmatrix}.
\end{equation}
Since $\alpha\mu>0$, then the equilibrium point is unstable for any $\alpha,\beta,\mu,k>0$.
\end{proof}

We prove below that the saturating equilibrium point is unstable:
\begin{proposition}
 \blue{Let us consider the closed-loop system consisting of the internally positive system \eqref{eq:mainsystL} and the controller \eqref{eq:ic3}}. Assume further that the system \eqref{eq:mainsystL} satisfies Assumption \ref{hyp:1} and that $\mu\in(0,-CA^{-1}Bk\beta)$.  Then, the zero-equilibrium point \eqref{eq:saturating_eqpt_rat} is unstable for all $\alpha,\beta,k>0$.
\end{proposition}
\begin{proof}
The linearized system about that equilibrium point is given by
\begin{equation}
  \begin{bmatrix}
    \dot{\tilde{x}}(t)\\
    \dot{\tilde{v}}(t)
  \end{bmatrix}=\begin{bmatrix}
    A & Bk\\
    0 & -\alpha(\mu+CA^{-1}Bk\beta)
  \end{bmatrix} \begin{bmatrix}
    \tilde{x}(t)\\
    \tilde{z}(t)
  \end{bmatrix}.
\end{equation}
Since $\mu>-CA^{-1}Bk\beta$ by assumption, then the term  $-\alpha(\mu+CA^{-1}Bk\beta)$ is positive and the equilibrium point is unstable.
\end{proof}

We prove below that the positive equilibrium point is locally exponentially stable provided that a condition is met:
\begin{proposition}\label{prop:stab_positive_2}
   \blue{Let us consider the closed-loop system consisting of the internally positive system \eqref{eq:mainsystL} and the controller \eqref{eq:ic3}}. Assume further that the system \eqref{eq:mainsystL} satisfies Assumption \ref{hyp:1} and that $\mu\in(0,-CA^{-1}Bk\beta)$. Then, the equilibrium point \eqref{eq:positive_eqpt_rat} is locally exponentially stable provided that $k\alpha\in(0,\bar{\xi})$ where
   \begin{equation}
      \bar{\xi}_\infty:=\sup\left\{\nu>0:\ \begin{bmatrix}
        A & B\\
          -\dfrac{\nu\beta}{4}C & 0
      \end{bmatrix}\ \textnormal{Hurwitz}\right\}.
    \end{equation}
    When the matrix is Hurwitz stable for all $\nu>0$, then $\bar{\xi}_\infty=\infty$. Moreover, we have that
    \begin{equation}
        \bar{\xi}_\infty=\dfrac{4\bar{\alpha}_\infty(\mu)\mu}{g\beta}.
    \end{equation}
\end{proposition}


\begin{proof}
The linearized system about that equilibrium point is given by
\begin{equation}
  \begin{bmatrix}
    \dot{\tilde{x}}(t)\\
    \dot{\tilde{v}}(t)
  \end{bmatrix}=\begin{bmatrix}
    A & Bk\\
    -\dfrac{\alpha}{\beta}v^*(\beta-v^*)C &0
  \end{bmatrix} \begin{bmatrix}
    \tilde{x}(t)\\
    \tilde{z}(t)
  \end{bmatrix}
\end{equation}
where $v^*=-\mu/CA^{-1}Bk$. Since $0<\mu<-CA^{-1}Bk\beta$, then $\beta-v^*>0$ and therefore $    -\dfrac{\alpha}{\beta}v^*(\beta-v^*)<0$. The lower-left term lies in the interval $[-\alpha\beta/4, 0)$, therefore if $k$ and $\alpha$ are chosen such that the matrix
\begin{equation}
\begin{bmatrix}
    A & Bk\\
    -\dfrac{\alpha\beta}{4}C &0
  \end{bmatrix}
\end{equation}
is Hurwitz stable, then we know that for any $0<\mu<-CA^{-1}Bk\beta$, the equilibrium will be locally asymptotically stable. A coordinate change yields the result.
\end{proof}

\subsection{Examples}

\noindent\textbf{Gene expression.} We consider back the gene expression system \eqref{eq:geneexpression} and we can compute $\bar{\xi}_\infty$ using the formula in Proposition \ref{prop:stab_positive_2} to get that
\begin{equation}
  \bar\xi_\infty=\dfrac{4\gamma_1\gamma_2(\gamma_1+\gamma_2)}{\beta k_2}.
\end{equation}

\begin{figure}[H]
  \centering
  \includegraphics[width=0.7\textwidth]{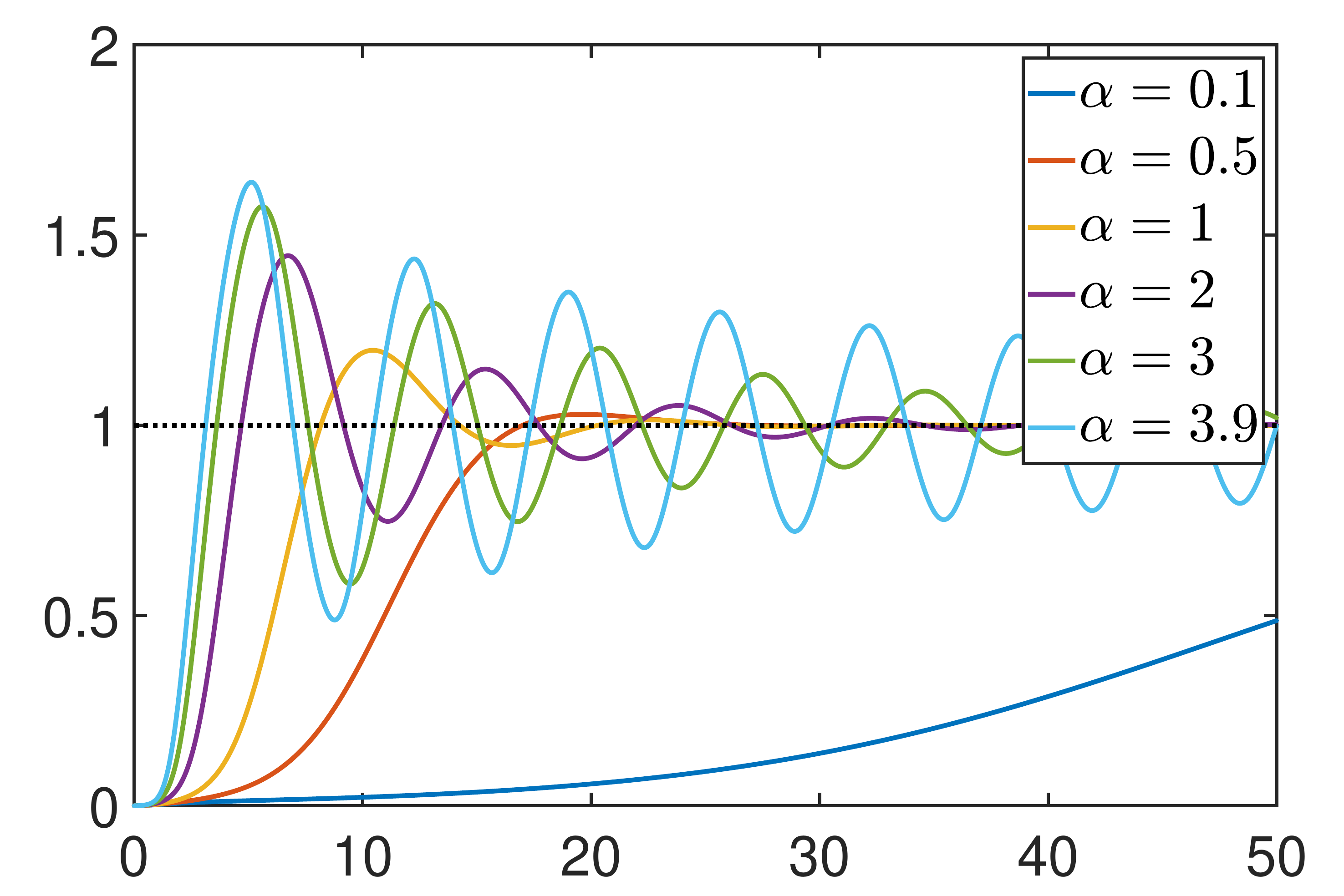}
  \caption{Output of the controlled gene expression system with protein maturation \eqref{eq:example_th} with the logistic integral controller \eqref{eq:ic3} for $k=1$ and several values for $\alpha$.}\label{fig:output_linear_sigmoidal}
\end{figure}

\begin{figure}[H]
  \centering
  \includegraphics[width=0.7\textwidth]{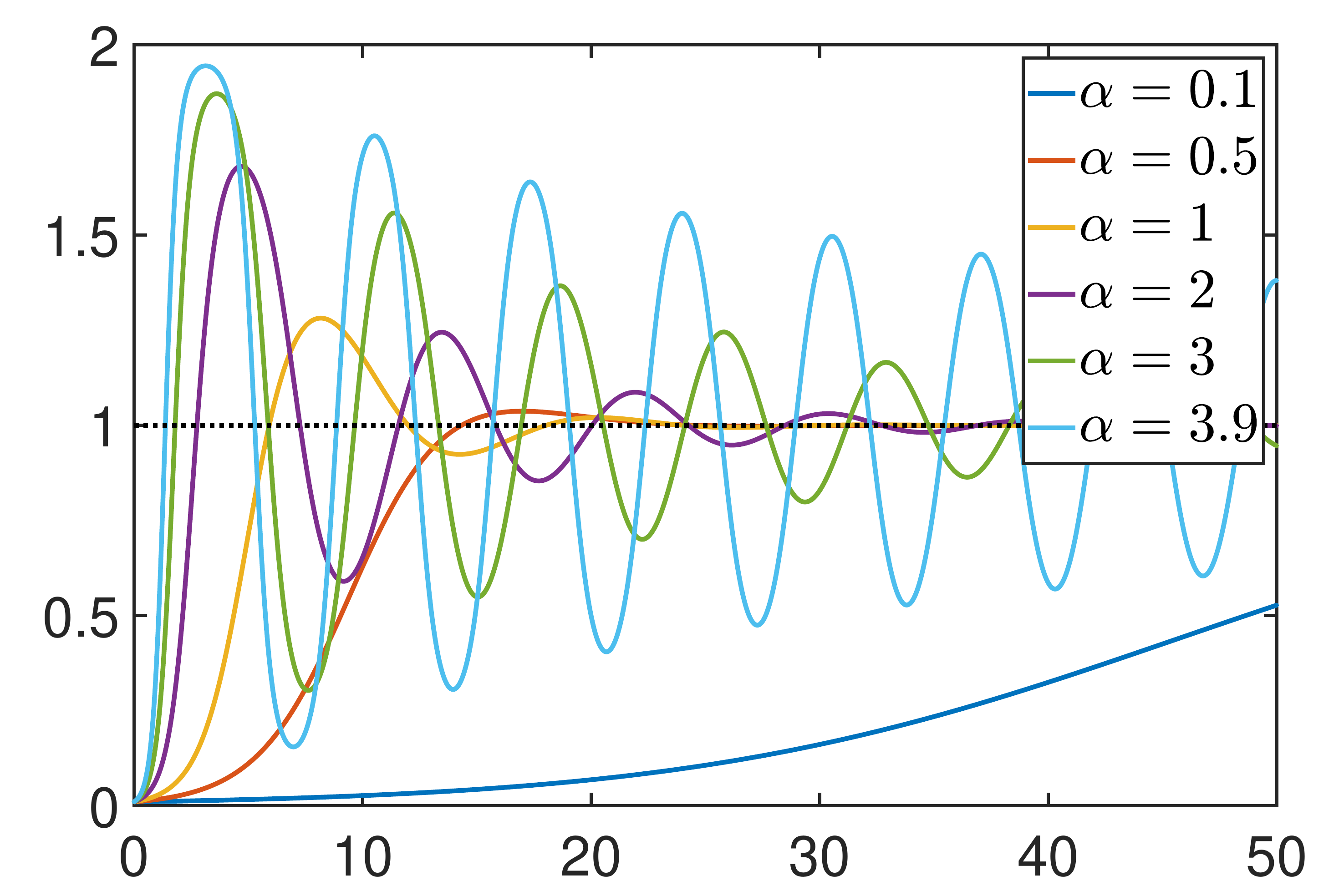}
  \caption{State of the integrator of the controlled gene expression system with protein maturation \eqref{eq:example_th} and the logistic integral controller \eqref{eq:ic3} for $k=1$ and several values for $\alpha$.}\label{fig:v_linear_sigmoidal}
\end{figure}

\begin{figure}[H]
  \centering
  \includegraphics[width=0.7\textwidth]{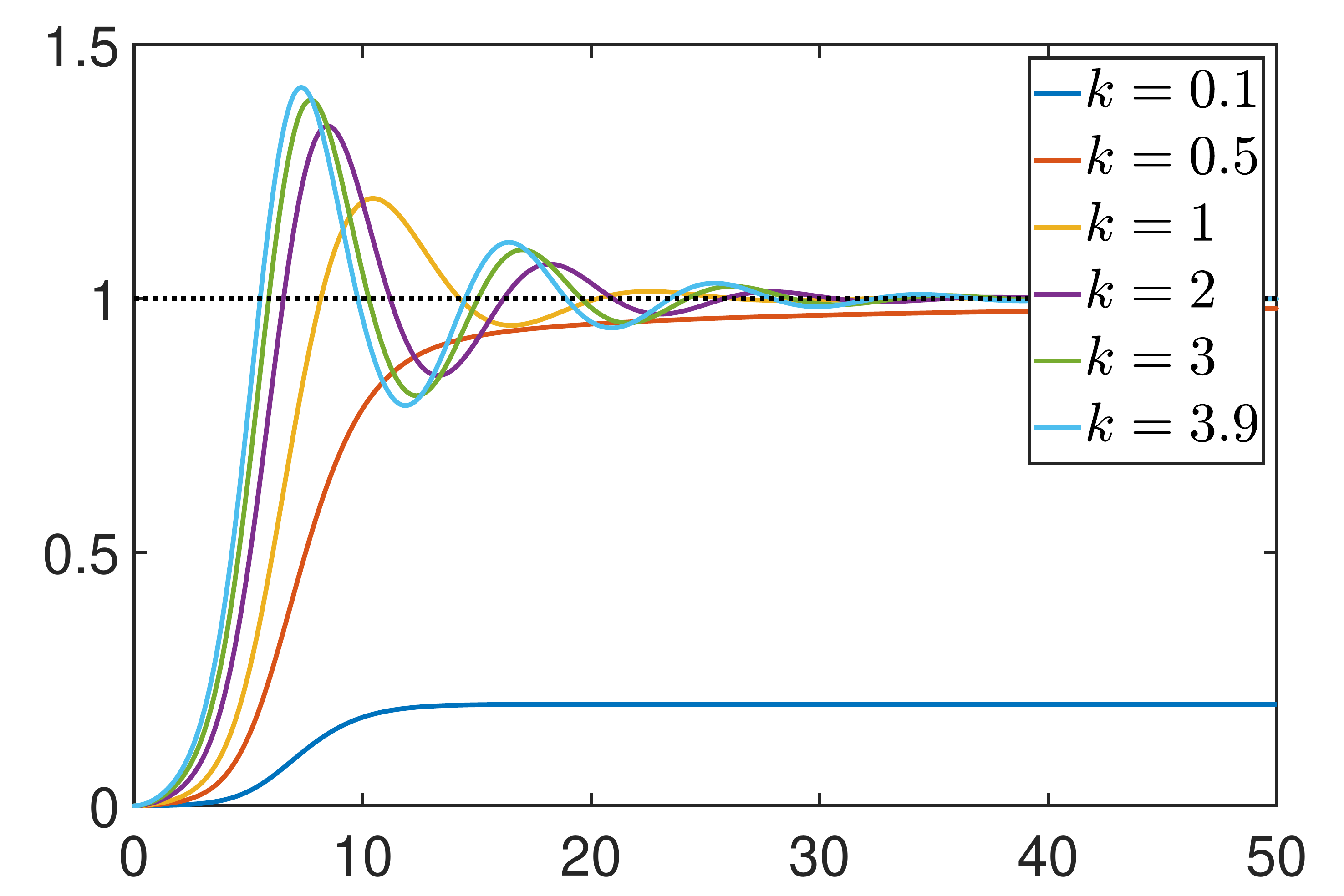}
  \caption{Output of the controlled gene expression system with protein maturation \eqref{eq:example_th} with the logistic integral controller \eqref{eq:ic3} for $\alpha=1$ and several values for $k$.}\label{fig:output_linear_sigmoidal2}
\end{figure}

\noindent\textbf{Gene expression with protein maturation.} Let us consider back the gene expression network with protein maturation \eqref{eq:example_th}. We can apply the same method as for the exponential integral controller to compute the value of $\bar \xi_\infty$. We notably obtain
%
\begin{equation}
  \bar{\xi}_\infty=\dfrac{4}{\beta}\dfrac{\gamma_1\gamma_3(\gamma_1+ \gamma_3)(\gamma_2 + k_3)(\gamma_1 + \gamma_2 + k_3)(\gamma_2 + \gamma_3 + k_3)}{k_2k_3(\gamma_1 + \gamma_2 + \gamma_3 + k_3)^2}.
\end{equation}

\section{Discussion}\label{sec:discussion}

Two approaches have been presented for the derivation of positive integral controllers. The first one is based on the consideration of an antithetic integral controller previously introduced in \cite{Briat:15e} in the context of the integral control of biochemical networks. The second one is based on the use of regularizing functions placed between the output of the controller and the input of system. Several local stability conditions as well as other qualitative results have been obtained.

In the present paper, internally positive systems have been extensively considered. However, the main underlying requirement was that the input and output signals be nonnegative at all times, which is a weaker property than internal positivity. This property is  referred to as \emph{external positivity} in the literature and is more difficult to characterize than internal positivity; see e.g. \cite{Bell:10}. In this respect, it might be possible that some of the obtained results extend to that class of systems. A less straightforward extension of those results would be the consideration of linear systems having a system matrix having a multiple or not eigenvalue at zero or, in the nonlinear setting, Jacobian matrix having eigenvalues on the imaginary axis.

 Other interesting extensions of this work include elucidating the question of whether the equilibrium point of the closed-loop system is also globally asymptotically stable under the same conditions. An obvious, yet tedious, approach would be the construction of a Lyapunov function. So, alternative, more generic, approaches may be interesting to consider instead. The application of the logistic  regularizing function to the control of nonlinear polynomial or rational systems subject to input saturation would also be of interest as this would allow for the use of efficient algorithms relying, for instance, on sum of squares programming. All these problems are left for future research.



\end{document}